\documentclass[12pt]{amsart}

\usepackage{amssymb}
\usepackage{fancyhdr}
\usepackage{amsmath}
\usepackage{amsfonts}
\usepackage{mathrsfs}
\usepackage[all]{xy} 
\usepackage{color}

\newcommand{\disk}{\ensuremath{\mathbb{D}} } 
\newcommand{\sphere}{\overline{\Bbb{C}}} 
\newcommand{\riem}{\Sigma}  
\renewcommand{\Bbb}[1]{\ensuremath{\mathbb{#1}}}

\theoremstyle{plain}
        \newtheorem{theorem}{Theorem}[section]
        \newtheorem{lemma}[theorem]{Lemma}
        
        \newtheorem{corollary}[theorem]{Corollary}

\theoremstyle{definition}
        \newtheorem{definition}[theorem]{Definition}

\theoremstyle{remark}
    \newtheorem{remark}[theorem]{Remark}

\numberwithin{equation}{section} 
\numberwithin{figure}{section} 

\usepackage{fullpage}

\title{Schiffer comparison operators and approximations on Riemann surfaces bordered by quasicircles}

\author{Eric Schippers}
\author{Mohammad Shirazi}
\author{Wolfgang Staubach}

\begin{document}
	\maketitle
	
\begin{abstract}  
 We consider a compact Riemann surface $R$ of arbitrary genus, with a finite number of non-overlapping quasicircles, which separate $R$ into two subsets: a connected Riemann surface $\riem$, and the union $\mathcal{O}$ of a finite collection of simply-connected regions.  We prove that the Schiffer integral operator mapping the Bergman space of anti-holomorphic one-forms on $\mathcal{O}$ to the Bergman space of holomorphic forms on $\riem$ is an isomorphism.  We then apply this to prove versions of the Plemelj-Sokhotski isomorphism and jump decomposition for such a configuration.  Finally we obtain some approximation theorems for the Bergman space of one-forms and Dirichlet space of holomorphic functions on $\riem$ by elements of Bergman space and Dirichlet space on fixed regions in $R$ containing $\riem$.
\end{abstract}

\begin{section}{Statement of the main theorems}

\begin{subsection}{Preliminaries}
 Let $\riem$ be a Riemann surface.  
 Define the pairing of one-forms on $\riem$
\begin{equation} \label{eq:form_inner_product}
 (\omega_1,\omega_2) = \frac{1}{2} \iint_\riem \omega_1 \wedge \ast \overline{\omega_2}. 
\end{equation}  
 Let $A(\riem)$ denote the set of holomorphic one-forms on $\riem$ for which this pairing is finite.  Similarly, let $\overline{A(\riem)}$ denote the set of anti-holomorphic one-forms for which (\ref{eq:form_inner_product}) is finite.  Denote the set of harmonic one-forms such that this pairing is finite by $A_{\text{harm}}(\riem)$.
 We have the orthogonal decomposition 
 \[  A_{\text{harm}}(\riem) = A(\riem) \oplus \overline{A(\riem)}.  \]
 The subscript $e$ will denote the subset of exact forms; e.g. $A(\riem)_e$, $A_{\text{harm}}(\riem)_e$ etc.  
 
 We also define the Dirichlet spaces.  Let 
 \begin{align*}
   \mathcal{D}_{\text{harm}}(\riem) & = \{ h: \riem \rightarrow \mathbb{C} \ \text{harmonic} \,:\, dh \in A_{\text{harm}}(\riem)  \} \\
   \mathcal{D}(\riem) & = \{ h: \riem \rightarrow \mathbb{C} \ \text{holomorphic} \,:\, \partial h \in A(\riem)  \} \\
   \overline{\mathcal{D}(\riem)} & = \{ h: \riem \rightarrow \mathbb{C} \ \text{anti-holomorphic} \,:\, \overline{\partial} h \in \overline{A(\riem)}  \}. 
 \end{align*}
 For a point $q \in \riem$, the subscript $q$ will denote the subset of functions vanishing at $q$; e.g. $\mathcal{D}_{\text{harm}}(\riem)_q$, etc.   
 
 Denote complex conjugation of functions $h$ and forms $\alpha$ by $\overline{h}$ and $\overline{\alpha}$.  
 Of course, $\overline{\mathcal{D}(\riem)}$ consists of the set of complex conjugates of elements of $\mathcal{D}(\riem)$, justifying the notation.  
 The notation $\overline{A(\riem)}$ is similarly justified.  
  
 By a {{conformal map}}, we mean a holomorphic map which is a biholomorphism onto its image.  To define a {{quasiconformal map}} we  first need the notion of {\it{Beltrami differential}} on $\Sigma$, which
 is a $(-1,1)$-differential $\omega$ on $\Sigma$, i.e., a differential given
 in a local biholomorphic coordinate $z$ by $\mu(z)d\bar{z}/dz$, such
 that $\mu$ is Lebesgue-measurable in every choice of coordinate and
 $||\mu||_\infty <1$.  Quasiconformal maps are by definition solutions to the {{Beltrami equation}}, i.e.  the differential equation given in local
 coordinates by $\overline{\partial} f = \omega \partial f$ where $\omega$ is a Beltrami differential on $\Sigma$.
 Let $\mathbb{C}$ denote the complex plane and $\sphere$ denote the Riemann sphere. By a quasicircle in the plane, we mean the image of the unit circle
 $\mathbb{S}^1 = \{ z \in \mathbb{C} \,:\, |z|=1 \}$ under a quasiconformal mapping of the plane.  By a quasicircle $\Gamma$ in $\riem$, we mean a simple closed curve such that there is a conformal map $\phi:U \rightarrow \mathbb{C}$ such that $U$ is an open neighbourhood of $\Gamma$ and $\phi(\Gamma)$ is a quasicircle in $\mathbb{C}$ in the sense above.
 
  Let $R$ be a compact surface.  Fix points $z$, $q$, and $w_0 \in R$.  Following for example H. Royden \cite{Royden}, we define Green's function of $R$ to be the unique function 
 $g(w,w_0;z,q)$ such that 
 \begin{enumerate}
  \item $g$ is harmonic in $w$ on $R \backslash \{z,q\}$;
  \item for a local coordinate $\phi$ on an open set $U$ containing $z$, $g(w,w_0;z,q) + \log| \phi(w) -\phi(z) |$ is harmonic 
   for $w \in U$;
  \item for a local coordinate $\phi$ on an open set $U$ containing $q$, $g(w,w_0;z,q) - \log| \phi(w) -\phi(q) |$ is harmonic 
   for $w \in U$;
  \item $g(w_0,w_0;z,q)=0$ for all $z,q,w_0$.  
 \end{enumerate}
 It can be shown that $g$ exists, and is uniquely determined by these properties.  The normalization at $w_0$ is inconsequential, because it can be shown that $\partial_w g(w,w_0;z,q)$ and $\overline{\partial}_w g(w,w_0;z,q)$ are independent of $w_0$.  Thus we leave $w_0$ out of the notation for $g$.  Also, $g$ is harmonic in both $w$ and $z$.  
 
 Let $R$ be compact as above, and $\riem \subset R$ be an open, proper and connected subset which we treat as a Riemann surface.  For such surfaces we have a different notion of Green's function.  We say that $\riem$ has a Green's function if there is a harmonic function $g(z,w)$ on $\riem$ such that 
 \begin{enumerate}
  \item for a local coordinate $\phi$ on an open set $U \subset \riem$ containing  $w$, $g(z;w) + \log| \phi(z) -\phi(w) |$ is harmonic in $z$ on $U$;
  \item $\lim_{z \rightarrow p} g(z,w) = 0$ for all $p \in \partial \riem$ and $w \in \riem$.       
 \end{enumerate}
 It can be shown that $g$ is also harmonic in $w$.  
 
 It is always understood that we use the first type of Green's function for compact surfaces, and the second type for open proper connected subsets.  When necessary, we distinguish between Green's functions of different surfaces with a subscript, e.g. $g_R$ or $g_\riem$.  
\end{subsection}
\begin{subsection}{Statement of results} 
  Let $R$ be a compact surface and $\mathcal{O}$ be an open subset.  In this paper, we will always assume that $\mathcal{O} = \Omega_1 \cup \cdots \cup \Omega_n$ where $\Omega_k$ are simply-connected domains for $k=1,\ldots,n$, each bounded by a quasicircle, whose closures are pairwise disjoint.  Let $\riem$ be the complement of the closure of $\mathcal{O}$ in $R$.  

   We define the integral operators  
  \begin{align}  \label{eq:T_general_definition}
   T(\mathcal{O},\riem) : \overline{A(\mathcal{O})} & \rightarrow A(\riem) \nonumber \\
   \overline{\alpha} & \mapsto \frac{1}{\pi i} \iint_{\mathcal{O}} \partial_z \partial_w g_{R}(w;z,q) \wedge_w \overline{\alpha(w)} 
  \end{align}
  where $z \in \riem$.
  We refer to the kernel of this integral operator $L_R(z,w) = \partial_z \partial_w g_{{R}}(w;z,q)/(\pi i)$ as the Schiffer kernel.  Note that the above case includes the possibility that the domain $\mathcal{O}$ is connected; we will frequently use this case, for example when restricting to one of the connected components $\Omega_k$ of $\mathcal{O}$.  In that case we will denote the complement of the closure of $\Omega_k$ by $\Omega_k^*$.  The fact that $T(\mathcal{O},\riem)$ is bounded and has codomain $A(\riem)$ will be justified ahead.
  
  Let 
  \[  V = \left\{ \overline{\alpha} \in \overline{A(\mathcal{O})}: \iint_\mathcal{O} \beta \wedge \overline{\alpha} =0 \ \ \forall \beta \in A(R) \right\}.    \]
  that is, $V$ is the orthogonal complement of $\left. \overline{A(R)} \right|_{\mathcal{O}}$ in $\overline{A(\mathcal{O})}$.
  Our first result is the following:  
  \begin{theorem}  \label{th:general_T_is_isomorphism}  Let $R$ be a compact   Riemann surface, and let $\mathcal{O} = \Omega_1 \cup \cdots \cup \Omega_n$ where $\Omega_k$ are simply-connected domains in $R$ bounded by quasicircles $\Gamma_k$ for $k=1,\ldots,n$.  Assume that the closures of $\Omega_k$ are pairwise disjoint. Let $\riem$ be the complement of the closure of $\mathcal{O}$ in $R$.  Then the restriction of $T(\mathcal{O},\riem)$ to $V$ is a bounded isomorphism onto $A(\riem)_e$.  
  \end{theorem}
  
   This generalizes one direction of a result of V.V. Napalkov and R.S. Yulmukhametov \cite{Nap_Yulm}, which says that in the case that $n=1$ and $R = \sphere$, the Schiffer operator is an isomorphism if and only if the domain $\Omega$ is bounded by a quasicircle.  This is closely related to a result proven by Schippers and Staubach \cite{Schippers_Staubach_scattering} which shows that the jump decomposition in the plane results in a bounded isomorphism if and only if the curve is a quasicircle, and also a result of Y. Shen \cite{ShenFaber} which shows that the boundary of $\Omega$ is a quasicircle if and only if a certain sequential Faber operator is a bounded isomorphism.    These results motivate the particular interest in Schiffer operators for regions bounded by quasicircles.  Schippers and Staubach proved this theorem \cite{Schiffer_comparison} in the case of a single quasicircle dividing a compact Riemann surface into two disjoint connected components.\\
  
We also give three applications of Theorem \ref{th:general_T_is_isomorphism}. 
 First, we prove a version of the Plemelj-Sokhotski jump formula for $\Gamma = \Gamma_1 \cup \cdots \cup \Gamma_n$ where $\Gamma_k$ are as in Theorem \ref{th:general_T_is_isomorphism}.  By $\mathcal{H}(\Gamma)$, we mean the set 
 of complex functions on $\Gamma$ whose restriction to $\Gamma_k$ is the boundary values of an element of $\mathcal{D}_{\text{harm}}(\Omega_k)$ in a sense which we refer to as ``conformally non-tangential (CNT)'' (see Section \ref{se:transmission} for the precise definition).   
Let \[ W =  \left\{ g \in
 \mathcal{D}_{\text{harm}}(\mathcal{O}):  \overline{\partial} g \in V \right\}.    \]
  \begin{theorem} \label{th:jump_proper}  Let $R$, $\Gamma$, $\Omega_k$ and $\riem$ be as in \emph{Theorem \ref{th:general_T_is_isomorphism}}, and  
  let $H \in \mathcal{H}(\Gamma)$ be such that its extension $h$ to $\mathcal{D}_{\mathrm{harm}}(\mathcal{O})$
is in $W$.  Fix $q \in \riem$.  There are unique $h_k \in \mathcal{D}(\Omega_k)$, $k=1,\ldots,n$, 
and $h_{\riem} \in \mathcal{D}(\riem)_q$ so that
  if $H_k$, $H_\riem$ are their \emph{CNT} boundary values, then on each curve $\Gamma_k$, $H=- H_\riem + H_k$.  These  are given by
  \[    h_k = \left. J_q(\Gamma) h \right|_{\Omega_k}  \]
  for $k=1,\ldots,n,$ and 
  \[  h_\riem = \left. J_q(\Gamma) h \right|_{\riem}.   \]
 \end{theorem}
 Here, $J_q(\Gamma)$ is an integral operator similar to the Cauchy integral, with integral kernel $-\partial_w g(w;z,q)/(\pi i)$.  See Section \ref{se:integral_operators} ahead for the precise definition, which involves approximations of the quasicircles by analytic curves.
 
 {It is classically known that there is such a jump decomposition for reasonably smooth curves and functions on Riemann surfaces; see \cite{Gakhov_book,Rodin_book}.  This was generalized to the case of a single quasicircle separating a compact Riemann surface into two components, and data in $\mathcal{H}(\Gamma)$ as above, in \cite{Schiffer_comparison}. A discussion of the literature can also be found there. The space $W$ corresponds to the classical condition for existence of a jump; see Section \ref{se:jump} ahead. } 
 
 The second application is an approximation theorem for Dirichlet spaces of holomorphic functions and Bergman spaces of holomorphic one-forms.
  
 \begin{theorem}  \label{th:dirichlet_density_squeeze} 
  Let $R$ be a compact Riemann surface and let $\riem, \Omega_k,\Gamma_k$ and $\riem', \Omega_k', \Gamma_k'$ each be as in  \emph{Theorem \ref{th:general_T_is_isomorphism}}.  Assume further that $\riem \subset \riem'$ and that the quasicircles $\Gamma_k'$ are isotopic to $\Gamma_k$ within the closure of $\Omega_k'$ for each $k=1,\ldots,n$.  
  
  If $\riem'' \subset R$ is any open set such that $\riem \subseteq \riem'' \subseteq \riem'$, then
  \begin{enumerate}
      \item the set of restrictions of elements of $\mathcal{D}(\riem'')$ to $\riem$ is dense in $\mathcal{D}(\riem)$
      \item  the set of restrictions of $A(\riem'')$ is dense in $A(\riem)$.
  \end{enumerate}
 \end{theorem}	

 Since one may always view $\riem'$ as embedded in its double, one can remove the mention of the outer surface and obtain
 \begin{corollary}  \label{co:embedding_in_double}
  Let $\riem'$ be a bordered Riemann surface whose boundary consists of $n$ curves $\Gamma_1',\ldots,\Gamma_n'$ homeomorphic to $\mathbb{S}^1$, whose double is compact.  Assume that $\riem \subset \riem'$ is an open set bordered by $n$ quasicircles $\Gamma_1,\ldots,\Gamma_n$, such that $\Gamma_k$ is isotopic to $\Gamma_k'$ in $(\riem' \cup \partial \riem') \backslash \riem$ for $k=1,\ldots,n$.  For any open set $\riem''$ such that $\riem \subseteq \riem'' \subseteq \riem'$, 
  \begin{enumerate}
      \item the set of restrictions of $\mathcal{D}(\riem'')$ to $\riem$ is dense in $\mathcal{D}(\riem).$
      \item the set of restrictions of elements of $A(\riem'')$ to $\riem$ is dense in $A(\riem).$
  \end{enumerate}
    
 \end{corollary}
 Here, note that we mean that the boundaries are borders \cite{Ahlfors_Sario}.\\
 
  These results should be compared to a result of N. Askaripour and T. Barron \cite{AskBar}, which says that if $D_1$ and $D_2$ are open subsets of a Riemann surface $R$ such that $D_1 \subseteq D_2$, and the lift to the universal cover (the disk $\disk$) of $D_1$ and $D_2$ are Carath\'eodory sets contained in a smaller disk, then restrictions of elements of the Bergman space $A(D_2)$ to $D_1$ are dense in $A(D_1)$.  As far as we know, their result is the first general result for nested Riemann surfaces, for $L^2$ approximability, {as opposed to uniform approximation, e.g. P. Gauthier and F. Sharifi \cite{GauthierSharifi_Luzin} (see also F. Sharifi \cite{Sharifi_thesis} for a literature review). \\
  
  The approach of Askaripour and Barron uses a lift to the universal cover and application of Poincar\'e series. It would be of great interest to obtain our approximation theorems by applying their methods.  Our approach here ultimately relies on sewing.}\\
 
 The third application involves another kind of operator which we now define.  Let $R$ be a compact Riemann surface, $\riem$ and $\riem'$ be Riemann surfaces such that $\riem \subset \riem' \subset R$ and such that $\mathrm{cl}\riem\subset \riem'$ (where the closure is with respect to the topology of $R$) and the inclusion maps from $\riem$ to $\riem'$ and $\riem'$ to $R$ are holomorphic.  
  
  \begin{align*}
 S(\riem,\riem'): A(\riem) & \rightarrow A(\riem') \\
  \alpha & \mapsto - \frac{1}{\pi i} \iint_{\riem}  \partial_z 
  \overline{\partial}_w g_{\riem'}(w;z,q)\wedge_w \alpha(w)
\end{align*}
The kernel of this integral operator 
\[  K_{\riem'}(z,w) = - \frac{1}{\pi i} \partial_z \overline{\partial}_w g_{\riem'}(w;z,q) \]
is the Bergman kernel of $\riem'$.  Note however that we integrate only over $\riem$ and not all of $\riem'$.  

We then have 
\begin{theorem} \label{th:Bergman_comparison_dense} Let $R$, $\riem$, and $\riem'$ be as  in Theorem \emph{\ref{th:dirichlet_density_squeeze}}.  Then $S(\riem,\riem')$ has trivial kernel and dense image.
\end{theorem}

 The operators $T(\mathcal{O},\riem)$ and $S(\riem,\riem')$ are special cases of what we call Schiffer comparison operators. 
 Note that the domain of integration of $T(\mathcal{O},\riem)$ is the subset $\mathcal{O}$ of $R$, and thus the operator  depends on both $R$ and $\mathcal{O}$.  Similarly, $S(\riem,\riem')$ depends on both $\riem'$ and $\riem$.  In general, we are interested in the extent to which information about the two surfaces $\mathcal{O}$ and $\riem$, or $\riem$ and $\riem'$, is reflected in the properties of the Schiffer operators.  
 
M. Schiffer and others \cite{BergmanSchiffer,Courant_Schiffer,Schiffer_first,Schiffer_Spencer} have investigated these comparison operators in many cases.  The  Riemann surface $R$ might be the Riemann sphere, or a subset of the plane bounded by analytic curves; while the subset $\mathcal{O}$ might be a multiply-connected planar domain or a subdomain of a compact surface $R$.

\end{subsection}
\end{section}
\begin{section}{The jump and Schiffer comparison operators}
\begin{subsection}{A Cauchy-type operator on compact surfaces and a Schiffer comparison operator}  \label{se:integral_operators}
 In this section we bring together various identities for the integral operators, and generalize some of them to the case of several boundary curves. These include expressions for the integral operators in terms of a kind of Cauchy-integral.\\ 
 
  We begin with the case of one boundary curve.  Let $R$ be a compact Riemann surface and let $\Gamma$ be a quasicircle, whose complement we assume to consist of two connected components $\Omega$ and $\riem$.  Let $g_\Omega$ denote the Green's function of $\Omega$, and for fixed $p \in \Omega$ and $s >0$ let $\Gamma^{p}_{s}$ be the level curves $\{ w: g_\Omega(w,p) =s\}$.  For $s$ sufficiently small, these are in fact analytic simple closed curves, and we endow them with a positive orientation with respect to $p$.  Fixing $q \in \riem$, we define
   \begin{align}  \label{eq:jump_definition}
     J_q(\Gamma) : \mathcal{D}_{\text{harm}}(\Omega) & 
      \rightarrow \mathcal{D}_{\text{harm}}(\Omega \cup \riem)_q
      \nonumber \\
     h & \mapsto - \lim_{s \searrow 0} 
        \frac{1}{\pi i}   \int_{\Gamma^{p}_s} \partial_w g_R (w;z,q) h(w)  
  \end{align}
  for $z \in R \backslash \Gamma$. This operator indeed takes   $\mathcal{D}_{\text{harm}}(\Omega)$ into $\mathcal{D}_{\text{harm}}(\Omega \cup \riem)$ by \cite[Corollary 4.3]{Schiffer_comparison}, where by  the latter we mean a function on the disjoint union which is harmonic on $\Omega$ and $\riem$.  
  It was furthermore shown that the operator is bounded and independent of $p$.  We may write the level curves $\Gamma_s^p$ in terms of a conformal map $f:\disk \rightarrow \Omega$ such that $f(0)=p$, as the images $f(\{ z: |z|=e^{-s}\})$ of circles centred at $0$. \\ 
  
   Recall the operator $T(\mathcal{O},\riem)$ defined by \eqref{eq:T_general_definition} in the introduction.  Specializing to the case that $\mathcal{O}$ consists of a single simply-connected domain $\Omega$, yields an operator which we denote by $T(\Omega, \riem)$.  It follows from \cite[Theorem 3.9]{Schiffer_comparison} that this operator is bounded.   We also define
  \begin{align*}
   T(\Omega,\Omega) : \overline{A(\Omega)} & \rightarrow A(\Omega) \\
   \overline{\alpha} & \mapsto \frac{1}{\pi i} \iint_{\Omega} \partial_z \partial_w g_{{R}}(w;z,q) \wedge_w \overline{\alpha(w)} 
  \end{align*} 
  where $z \in \Omega$, which is also bounded by \cite[Theorem 3.9]{Schiffer_comparison}.

  Finally define
  \begin{align*}
    S(\Omega,R):A(\Omega) & \rightarrow A(R) \\
    \alpha & \mapsto - \frac{1}{\pi i} \iint_\Omega \partial_z \overline{\partial}_{w} g_{{R}}(w;z,q) \wedge_w \alpha(w)
  \end{align*}
  for $z,q \in R$, which is bounded because the kernel function is globally bounded \cite{Schiffer_comparison}.
  
  The conjugate operator is defined by
  \begin{align} \label{eq:conjugate_definition}
    \overline{S}(\Omega,R): \overline{A(\Omega)} & \rightarrow \overline{A(R)} \nonumber \\
    \overline{\alpha} & \mapsto \overline{S(\Omega,R) \alpha}. 
  \end{align}
  Conjugates of $T$ operators are defined similarly.
  
  The operators $J_q (\Gamma), $ $T(\Omega, \riem),$ $T(\Omega, \Omega)$ and $S(\Omega,R)$ satisfy the identities \cite[Theorem 4.2]{Schiffer_comparison}
   \begin{align} \label{eq:derivative_J_identities}
    \partial J_q(\Gamma) h (z) & = - T(\Omega,\riem) \overline{\partial} h(z),  \    & z \in \riem \nonumber \\
    \partial J_q(\Gamma) h(z) & = \partial h(z) - T(\Omega,\Omega) \overline{\partial} h(z), \   
    & z \in \Omega \\
    \overline{\partial} J_q(\Gamma) h(z)& = \overline{S}(\Omega,R) \,\overline{\partial} h(z), \ 
    & z \in \Omega \cup \riem. \nonumber  
   \end{align}
  
  We would like to generalize these identities to the case of many boundary curves.  First, we make a general remark on notation.
  \begin{remark}[Direct sum notation] \label{re:natural_isomorphism}
  Let $\mathcal{O}$ be as in the introduction; that is, $\mathcal{O} = \Omega_1 \cup \cdots \cup \Omega_n$ for $\Omega_k$ simply-connected, bordered by quasicircles, with pairwise disjoint closures.  In that case, we have a natural isomorphism 
  \begin{align} \label{eq:obvious_isomorphism}
   A(\mathcal{O}) & \xrightarrow{\cong} \bigoplus_{k=1}^n A(\Omega_k) \nonumber \\
   \alpha & \mapsto \left( \left. \alpha \right|_{\Omega_1},\ldots, \left. \alpha \right|_{\Omega_n} \right).  
  \end{align}
  The inverse of this isomorphism is 
  \[  (\alpha_1,\ldots,\alpha_n) \rightarrow \sum_{k=1}^n \alpha_k \chi_k \]
  where $\chi_k$ are the characteristic functions of $\Omega_k$ for $k=1,\ldots,n$.  
   To avoid needless insertion of this isomorphism into every formula, for $\alpha \in A(\mathcal{O})$ say, we use the notation $\alpha_k = \left. \alpha \right|_{\Omega_k}$, and furthermore write without qualification
   \[  \alpha = (\alpha_1,\ldots,\alpha_n). \]
  Similarly, we have isomorphisms between $A(\mathcal{O})_e$ and $\bigoplus_{k=1}^n A(\Omega_k)_e$; $\mathcal{D}_{\text{harm}}(\mathcal{O})$ and $\bigoplus_{k=1}^n \mathcal{D}(\Omega_k)$; and so on. 
  \end{remark}
  
  With this convention in mind, for $(\overline{\alpha}_1,  \ldots, \overline{\alpha}_n)\in \bigoplus_{k=1}^n \overline{A(\Omega_k)}$, observe that $T(\mathcal{O},\riem)$ can be written 
\begin{align} \label{eq:TOsig_sum}
  [T(\mathcal{O}, \riem)(\overline{\alpha}_1, \dots, \overline{\alpha}_n)](z) &:= \frac{1}{\pi i}\sum_{k=1}^n \iint_{\Omega_k, w} \partial_z\partial_w g_{{R}}(w; z, q) \wedge \overline{\alpha}_k(w) \nonumber \\
&=\sum_{k=1}^n [T(\Omega_k, \Omega^*_k)\overline{\alpha}_k]_{\riem}(z)  
\end{align}
if $z\in \riem$ (recall that $\Omega_k^*$ is the complement of the closure of $\Omega_{{k}}$ in $R$).  Here for a set $A$, by $[\cdot]_{A}$ we mean the restriction to $A$.  The above expression shows that $T(\mathcal{O},\riem)$ is bounded as claimed in the introduction. 
For fixed $j=1, \dots, n$, we now define
\begin{align}  \label{eq:TOOmega_sum}
T(\mathcal{O}, \Omega_j): 
\overline{A(\mathcal{O})} & \rightarrow A(\Omega_j)  \nonumber\\
(\overline{\alpha}_1,\ldots,\overline{\alpha}_n) & \mapsto  \sum_{\substack{k=1\\ k\neq j}}^n [T(\Omega_k, \Omega^*_k)\overline{\alpha}_k]_{\Omega_j} +  T(\Omega_j, \Omega_j)\overline{\alpha}_j. 
\end{align} 
Again $[T(\Omega_k, \Omega^*_k)\overline{\alpha}_k]_{\Omega_j}$ is the restriction of $T(\Omega_k, \Omega^*_k)\overline{\alpha}_k$ to $\Omega_j$.  As above, boundedness follows directly from boundedness in the case for one boundary curve.  Finally define the bounded operator
\begin{align*}
 S(\mathcal{O},R) : A(\mathcal{O}) & \rightarrow A(R) \\
 \alpha & \mapsto \iint_{\mathcal{O}} K_R(\cdot,w) \wedge_w \alpha(w). 
\end{align*}
which again can be written as a sum of integrals over $\Omega_k$:  
\begin{equation} \label{eq:SOR_sum}
 S(\mathcal{O},R) \alpha =  \sum_{k=1}^n  S(\Omega_k,R) {\alpha}_k
\end{equation}
  
  We also set 
  \[  \Gamma = \Gamma_1 \cup \cdots \cup \Gamma_n \] 
  and define for $(h_1,\ldots,h_n) \in \mathcal{D}_{\text{harm}}(\mathcal{O})$ and $z \in R \backslash \Gamma$
 \begin{equation}\label{JE}
 [J_q(\Gamma)(h_1, \dots, h_n)](z)=\sum_{k=1}^n [J_q(\Gamma_k)h_k] (z).
 \end{equation}
 The identities (\ref{eq:derivative_J_identities}) can now be generalized as follows.
\begin{theorem}\label{DJO}
If $q$ is in $R \backslash \Gamma$ and $(h_1, \dots, h_n)\in \mathcal{D}_{\mathrm{harm}}(\mathcal{O})$ then 
\begin{align*}
\partial [J_q(\Gamma)(h_1, \dots, h_n)](z)&= -[T(\mathcal{O}, \Sigma)(\overline{\partial}h_1, \dots, \overline{\partial}h_n)](z), \hspace{2.55cm}z\in \Sigma,\\
\partial [J_q(\Gamma)(h_1, \dots, h_n)](z)&=-[T(\mathcal{O}, \Omega_j)(\overline{\partial}h_1, \dots, \overline{\partial}h_n)](z) + \partial h_j(z), \hspace{.6cm}z\in \Omega_j,\\
\overline{\partial} [J_q(\Gamma)(h_1, \dots, h_n)](z)&=[\overline{S}(\mathcal{O},R) (\overline{\partial}h_1, \dots, \overline{\partial}h_n)](z) \hspace{2.6cm} z\in R\backslash \Gamma.\\
\end{align*}
\end{theorem}
\begin{proof}
The first and third identities follow directly from  (\ref{eq:derivative_J_identities}), (\ref{eq:TOsig_sum}), and (\ref{eq:SOR_sum}).

Now let $z\in \Omega_j$ for fixed $j$; in this case for every $k\neq j$, $z\in  \Omega^*_k$.  Denoting by $J_q(\Gamma)_{\Omega_k^*}$ the operator obtained by restricting the output of $J_q(\Gamma)$ to $\Omega_k^*$, and similarly for $J_q(\Gamma)_{\Omega_j}$, we have
\begin{align*}
\partial [J_q(\Gamma)(h_1, \dots, h_n)]_{\Omega_j}&=\partial \sum_{k=1}^n [J_q(\Gamma_k)h_k]_{\Omega_j} (z)\\
&=\sum_{k\neq j} \partial [J_q(\Gamma_k)_{\Omega^*_k}h_k]_{\Omega_j}(z)+ \partial [J_q(\Gamma_j)_{\Omega_j}h_j](z)\\
&= \sum_{k\neq j} [T(\Omega_k, \Omega^*_k)\overline{\partial}h_k]_{\Omega_j}(z)+\partial h_j(z)+ T(\Omega_j, \Omega_j)\overline{\partial}h_j  (z)\\
&=T(\mathcal{O}, \Omega_j)(\overline{\partial}h_1, \dots, \overline{\partial}h_n)(z) + \partial h_j(z).
\end{align*}
\end{proof}

Throughout the paper, we will denote 
\[     J_q(\Gamma)_{\riem} h = \left. J_q h \right|_{\riem}, \ \ J_q(\Gamma)_{\mathcal{O}} h =  \left. J_q h \right|_{\mathcal{O}}, \ \ 
 J_q(\Gamma)_{\Omega_k} h =  \left. J_q h \right|_{\Omega_k}  \]
and so on, as above.  Thus for example the first two identities of the previous theorem can be expressed by 
 $\partial J_q(\Gamma)_\riem = -T(\mathcal{O},\riem) \overline{\partial}$ and $\partial J_q(\Gamma)_{\Omega_j} = \partial - T(\mathcal{O},\Omega_j) \overline{\partial}$.  


\end{subsection}
\begin{subsection}{Transmission of harmonic Dirichlet-bounded functions} 
\label{se:transmission}
  In this section, we consider certain operators, which take Dirichlet bounded functions on one region to Dirichlet bounded functions on another region sharing a portion of the boundary with the first, in such a way that the functions have the same boundary values.  These operators were studied in \cite{ Schippers_Staubach_general_transmission}; we briefly recall the necessary results.  
 
  We now explain the sense of boundary values, which we call ``conformally non-tangential boundary values'', or CNT boundary values.  All claims made here in the description of these boundary values are proven in \cite{Schippers_Staubach_general_transmission}.  For the purposes of this section we can assume that $\mathcal{O}=\Omega$ is a single simply-connected domain and $\Omega^*$ is the complement of its closure, as above.  
  
  For $p \in \Omega$, let
  \[  \Omega_{p,\epsilon} = \{ w \in \Omega : g_{\Omega}(w,p) < \epsilon \}.       \]
  Similarly for $p' \in \Omega^*$, let 
  \begin{equation} \label{eq:canonical_neighb_riem}
    \Omega^*_{p',\epsilon}= \{w \in \Omega^* : g_{\Omega^*}(w,p') < \epsilon  \}.  
  \end{equation}
  For $\epsilon$ sufficiently small, there is a conformal map $f:\mathbb{A} \rightarrow \Omega^*_{p',\epsilon}$ where $ \mathbb{A} =\{ z: e^{-\epsilon} < |z| < 1\}$, where $f$ extends to a homeomorphism from $\mathbb{S}^1$ onto $\Gamma$.  We can always assume that $f$ has a conformal extension to a neighbourhood of $|z|=e^{-\epsilon}$ so that the image of this curve is analytic.  In the case of $\Omega_{p,\epsilon}$, such a conformal map exists for all $\epsilon \in (0,1)$ and is in fact the restriction of the conformal map from $\disk$ to $\Omega$ taking $0$ to $p$.   In particular, $\Omega^*_{p',\epsilon}$ and $\Omega_{p,\epsilon}$ are doubly-connected domains.  
  
  Given a $u \in \mathcal{D}_{\text{harm}}(\Omega_{p,\epsilon})$ or $\mathcal{D}_{\text{harm}}(\Omega^*_{p',\epsilon})$, the non-tangential boundary values of $u \circ f$ exist except perhaps on a Borel set of logarithmic capacity zero.   The boundary value of $u$ at a point $\zeta \in \Gamma$ is defined to be the boundary value of $u \circ f$ at $f^{-1}(\zeta)$, where it exists. We call these boundary values the CNT boundary values. 
  
  We call the image of a Borel set of logarithmic capacity zero under $f$ a ``null set'' on $\Gamma$.  This definition can be shown to be independent on $p$ or $p'$, and the particular choice of map $f$.  Similarly, the CNT boundary values are independent of these choices.
  An important fact is that if $u_1,u_2 \in \mathcal{D}_{\text{harm}}(\Omega)$ have the same CNT boundary values except possibly on a null set, then they are equal.  Similarly for $u_1,u_2 \in \mathcal{D}_{\text{harm}}(\Omega^*)$.
  
  It is a much more subtle fact that a set on $\Gamma$ is null with respect to the conformal map onto $\Omega_{p,\epsilon}$ if and only if it is null with respect to the conformal map onto $\Omega^*_{p',\epsilon}$, in the case that $\Gamma$ is a quasicircle.  More is true: if $u$ is the CNT boundary values of an element of $\mathcal{D}_{\text{harm}}(\Omega)$ except possibly on a null set, then it is the CNT boundary values of a unique element of $\mathcal{D}_{\text{harm}}(\Omega^*)$ except on a null set, and the converse also holds.  
  
 This allows us to define an operator that we call the transmission operator.  Given $u \in \mathcal{D}_{\text{harm}}(\Omega)$, one may obtain boundary values of $u$ on $\Gamma$ in the CNT sense (see Remark \ref{re:CNT_definition} below).
  There exists a unique  element of $\mathcal{D}_{\text{harm}}(\Omega^*)$ 
  with the same CNT boundary values, 
  which we denote by $\mathfrak{O}(\Omega,\Omega^*) u$.  
  This defines a map 
  \[  \mathfrak{O}(\Omega,\Omega^*):  \mathcal{D}_{\text{harm}}(\Omega) \rightarrow \mathcal{D}_{\text{harm}}(\Omega^*),   \]
  and similarly a map 
  \[  \mathfrak{O}(\Omega^*,\Omega):  \mathcal{D}_{\text{harm}}(\Omega^*) \rightarrow \mathcal{D}_{\text{harm}}(\Omega).   \]
  By definition, these are inverses of each other.  It was shown \cite{Schippers_Staubach_general_transmission} that these are bounded with respect to the Dirichlet seminorm. We call these {\it transmission} operators, since in some sense they transmit a harmonic function through the quasicircle via a boundary value problem.
 
  Similarly, the CNT boundary values of any element of $\mathcal{D}_{\text{harm}}(\Omega_{p,\epsilon})$ are equal to the CNT boundary values of a unique element of $\mathcal{D}_{\text{harm}}(\Omega)$ up to a null set.  The converse is obviously true by simply restricting from $\Omega$ to $\Omega_{p,\epsilon}$; however, one does not obtain a unique element of $\mathcal{D}_{\text{harm}}(\Omega_{p,\epsilon})$.  The same claims are true for $\Omega^*$ and $\Omega^*_{p',\epsilon}$.  
  This allows us to define the following operator:
  \begin{equation} \label{eq:bouncedefinition}
     \mathfrak{G}(\Omega_{p,\epsilon},\Omega):\mathcal{D}_{\text{harm}}(\Omega_{p,\epsilon}) \rightarrow \mathcal{D}_{\text{harm}}(\Omega)   
  \end{equation}
  to take $u$ to the unique element of $\mathcal{D}_{\text{harm}}(\Omega)$ with the same CNT boundary values.  Similarly we define 
  \[   \mathfrak{G}(\Omega^*_{p',\epsilon}, \Omega^*):\mathcal{D}_{\text{harm}}(\Omega^*_{p',\epsilon}) \rightarrow \mathcal{D}_{\text{harm}}(\Omega^*).  \]
  It was shown in \cite{Schippers_Staubach_general_transmission} that these are bounded with respect to the Dirichlet seminorm. We call these {\it bounce} operators.
  
  The integral (\ref{eq:jump_definition}) could equally be defined using level curves in $\Omega^*$.  That is, let $p' \in \Omega^*$ and let $\Gamma^{p'}_{\epsilon}$ denote the level curves of Green's function, but now give them a negative orientation with respect to $\Omega^*$.  If we denote by $J_q(\Gamma,\Omega^*)$ the new operator defined using negatively-oriented level curves in $\Omega^*$, and by $J_q(\Gamma,\Omega)$ the original operator defined using positively-oriented level curves in $\Omega$, then 
  \cite[Theorem 4.10]{Schiffer_comparison}
  \begin{equation} \label{eq:J_both_sides}
      J_q(\Gamma,\Omega^*) \mathfrak{O}(\Omega,\Omega^*) = J_q(\Gamma,\Omega).
  \end{equation}
  The notation $J_q(\Gamma)$ will always refer to $J_q(\Gamma,\Omega)$.  The latter notation is used only when it is necessary to distinguish $J_q(\Gamma,\Omega)$ from $J_q(\Gamma,\Omega^*)$.  
  
  We will also use the following notation for $h \in \mathcal{D}(\Omega_{p,\epsilon})$.
  \begin{equation} \label{eq:Jprime_definition}
    J_q(\Gamma)'h(z) = - \frac{1}{\pi i} \lim_{s \searrow 0} \int_{\Gamma^p_s}
    \partial_w g_R (w;z,q) h(w). 
  \end{equation}
  Although the integral is the same, the prime is included to distinguish it from the operator $J_q(\Gamma)$, which has a different domain. 
  For any one-form in $A(\Omega_{p,\epsilon})$ and $h \in \mathcal{D}_{\text{harm}}(\Omega_{p,\epsilon})$, we have \cite[Theorem 4.8]{Schiffer_comparison}
  \begin{equation}  \label{eq:forms_same_average}
    \lim_{s \searrow 0} \int_{\Gamma^p_{s}} 
     \alpha  h    = \lim_{s \searrow 0} \int_{\Gamma^p_{s}} 
     \alpha\, \mathfrak{G}(\Omega_{p,\epsilon},\Omega) h.    
  \end{equation}
  Similarly \cite[Theorem 4.9]{Schiffer_comparison}
  \begin{equation} \label{eq:J_J'_same}
   J_q(\Gamma)' h = J_q(\Gamma) \mathfrak{G}(\Omega_{p,\epsilon},\Omega) h. 
  \end{equation}

  Finally, we define $\mathcal{H}(\Gamma)$ to be the set of complex-valued functions defined on $\Gamma$, which are the CNT boundary values of an element $h \in \mathcal{D}_{\text{harm}}(\Omega)$, modulo the following equivalence relation.
  We say that $h_1 \sim h_2$ if $h_1$ and $h_2$ are equal except possibly on a null set.  {We will continue to treat equivalence classes as functions in the customary way.}
  
  \begin{remark} \label{re:CNT_definition} 
  Except for Theorem \ref{th:jump_proper}, the precise meaning of ``conformally non-tangential'' is not directly relevant to the paper; the identities (\ref{eq:J_both_sides}), (\ref{eq:forms_same_average}), and (\ref{eq:J_J'_same}) above are logically sufficient to obtain the results here.  Of course, the meaning is helpful for an intuitive understanding of several theorems and proofs.
  \end{remark}  
 
\end{subsection}
\begin{subsection}{Density theorems for the image of $\mathfrak{G}$}

 In this section we prove some preliminary density theorems.
 	
Let $R$ be a compact Riemann surface, and $\mathcal{O}$, $\Omega_k$, and $\Gamma_k$, for $k=1,\ldots,n$ be as above.

Now {recall that}
\[ W =  \left\{ (g_1,\ldots, g_n) \in \oplus_{k=1}^n 
 \mathcal{D}_{\text{harm}}(\Omega_k) : \sum_{k=1}^n \iint_{\Omega_k} \alpha \wedge \overline{\partial} g_k =0 \ \ \forall  \alpha \in A(R) \right\}    \]
where we have made use of the isomorphism (\ref{eq:obvious_isomorphism}) to write $W$ in terms of the restrictions $g_k$ to $\Omega_k$. 
We also denote

\[ W' =  \left\{ (\overline{h}_1,\ldots, \overline{h}_n) \in \oplus_{k=1}^n 
\overline{\mathcal{D}(\Omega_k)} : \sum_{k=1}^n \iint_{\Omega_k} \alpha \wedge \overline{\partial} \overline{h}_k =0 \ \ \forall   \alpha \in A(R) \right\}.    \]
Because $\Omega_{k}$ is simply connected, we may decompose any $g$ as $g = e + \overline{h}$ where $e$ has only holomorphic components and $\overline{h}$ has only anti-holomorphic components.  Thus one may define $W$ equivalently to be the set of elements $g$ of $\oplus_{k=1}^n 
\mathcal{D}_{\text{harm}}(\Omega_k)$ whose anti-holomorphic component is in $W'$.

Note that 
\begin{equation}  \label{eq:jump_condition_doubleintegral}
  (\overline{h}_1,\ldots, \overline{h}_n ) \in W' 
  \Leftrightarrow  \sum_{k=1}^n \lim_{s \searrow 0} \int_{\Gamma_{s}^{p_k}} \overline{h}_k \alpha =0 \ \ \forall
  \alpha \in A(R)     
\end{equation}
by Stokes' theorem.

\begin{lemma}\label{CDJO} 
If $(h_1, \dots, h_n)\in W$ then $J_q(\Gamma)(h_1, \dots, h_n)$ is a holomorphic function on $R\backslash \Gamma$. 
\end{lemma}
\begin{proof}
We need to show that $\overline{\partial} [J_q(\Gamma)(h_1, \dots, h_n)](z)=0$ for every $z\in R\backslash \Gamma$. We use Theorem \ref{DJO} to derive the first identity below.
\begin{equation*}
\overline{\partial} [J_q(\Gamma)(h_1, \dots, h_n)](z)=[\overline{S}(\mathcal{O},R) (\overline{\partial}h_1, \dots, \overline{\partial}h_n)](z)\\
=\overline{\sum_{k=1}^n \iint_{\Omega_k, w} K_R(z, w)\wedge \partial{\overline{h}}_k(w)}. 
\end{equation*}
The last integral is zero since $K_R(z,\cdot) \in \overline{A(R)}$ for each fixed $z$ and $(\overline{\partial}h_1, \dots, \overline{\partial}h_n)\in V$.
\end{proof}

For fixed choice of $p_k \in \Omega_k$, $k=1,\ldots,n$, let $\Omega_{k,p_k,\epsilon}$ be the domains bounded by level curves of Green's function as in the previous section.  Define now the spaces
\[ X_{\epsilon} = \left\{ (u_1,\ldots,u_n) \in \oplus_k \mathcal{D}(\Omega_{k,p_k,\epsilon}) \,:\,  \sum_{k=1}^n \int_{\gamma_k}  u_k \alpha = 0  \ \ \forall \alpha \in A(R)  \right\}. \]
Here,  for $k=1,\ldots,n$, $\gamma_k$ is any choice of simple closed analytic curve in $\Omega_{k,\epsilon,p_k}$ which is isotopic to $\Gamma^{p_k}_{\epsilon}$ within the closure of $\Omega_{k,p_k,\epsilon}$.   

Since all theorems hold for any choice of $p_k$, we will remove the points from the notation for the domains.  That is, we will denote $\Omega_{k, p_k,\epsilon}$ by $\Omega_{k,\epsilon}$. 

Recall the definition (\ref{eq:bouncedefinition}) of the bounce operator.  Denote 
\[\oplus_k \mathfrak{G}(\Omega_{k,\epsilon},\Omega_k) (u_1,\ldots,u_n) = 
\left(\mathfrak{G}(\Omega_{1,\epsilon},\Omega_1)u_1,\ldots,\mathfrak{G}(\Omega_{n,\epsilon},\Omega_n)u_n \right).\]
We then have the following theorem.
\begin{theorem}
	Let $\gamma_k$, $k=1, \ldots,n$ be analytic Jordan curves in $\Omega_{k,\epsilon}$ respectively, such that each $\gamma_k$ is isotopic to $\Gamma^{p_k}_{\epsilon}$ within the closure of $\Omega_{k,\epsilon}$.  Given any 
	\[ (u_1,\ldots,u_n) \in \oplus_k\mathcal{D}(\Omega_{k,\epsilon})  \]
	we have that
	\[ \oplus_k \mathfrak{G}(\Omega_{k,\epsilon},\Omega_k) (u_1,\ldots, u_n)
	 \in W \ \  \Leftrightarrow \ \ \sum_{k=1}^n \int_{\gamma_k} u_k \alpha =0   \ \ \forall \alpha \in A(R). \]
	 
	 That is, $(u_1,\ldots,u_n) \in X_\epsilon$ if and only if $\oplus_k \mathfrak{G}(\Omega_{k,\epsilon},\Omega_k) 
	 (u_1,\ldots,u_n) \in W$.    
\end{theorem}
\begin{proof}
 This follows directly from (\ref{eq:forms_same_average}).
\end{proof}

We have the following two theorems.
\begin{theorem} \label{th:extension_collar}
 For any $(u_1,\ldots,u_n) \in X_{\epsilon}$,  the restriction of 
 \[  \sum_{k=1}^n J_q(\Gamma_k)' u_k \]
 to $\riem$  
 extends to a function which is holomorphic on $\mathrm{cl}\, \riem \cup \Omega_{1,\epsilon} \cup \cdots \cup \Omega_{n,\epsilon}$.  
\end{theorem}
\begin{proof}  Choose $0< r < \epsilon$.  Let $\Gamma_{k,r}$ denote the boundary of $\Omega_{k, r}$.  Let $B_r$ be the region containing $\riem$ and bounded by $\cup_{k=1}^n \Gamma_{k,r}$.
 Applying Royden \cite[Proposition 6]{Royden}, together with the explicit integral formula given there, we see that 
 \[  - \sum_{k=1}^n \frac{1}{\pi i} \int_{\Gamma_{k, r}} \partial_w g(w;z,q) u_k(w)   \]
 defines a holomorphic function on $B_r$ .
 But this integral is independent of $r$ and thus equals the limiting integral (\ref{eq:Jprime_definition}).  Since every point in $\mathrm{cl}\, \riem \cup \Omega_{1,\epsilon} \cup \cdots \cup \Omega_{n,\epsilon}$ is contained in some $B_r$ this proves the theorem.
\end{proof}

\begin{corollary}  \label{co:holo_extension}
	For any $(u_1,\ldots,u_n) \in X_\epsilon$, the restriction of  
	\[  \sum_{k=1}^n J_q(\Gamma_k)  \mathfrak{G}(\Omega_{k,\epsilon},\Omega_k) u_k    \]
	to $\riem$ has a holomorphic extension to $\mathrm{cl} \, \riem \cup \Omega_{1,\epsilon} \cup \cdots \cup \Omega_{n,\epsilon}$. 
\end{corollary}	
\begin{proof}
	This follows immediately from 
	 Theorem \ref{th:extension_collar} and (\ref{eq:J_J'_same}).  
\end{proof}
\begin{remark} \label{re:ignore_holomorphic}
    Fix $q \in \riem$.
	Observe that if $(g_1,\ldots,g_n)$ and $(\hat{g}_1,\ldots,\hat{g}_n)$ are in $W$, and $g_k - \hat{g}_k$ are holomorphic in $\Omega_k$ for all $k$, then
	\[ \left.  \sum_{k=1}^n J_q(\Gamma_k) g_k \right|_{\riem} = \left. \sum_{k=1}^n J_q(\Gamma_k) \hat{g}_k  \right|_{\riem}.  \]
	
To see this, by (\ref{eq:derivative_J_identities}) we have that 
	\[  \partial J_q(\Gamma_k) (g_k - \hat{g}_k)
	 = - T(\Omega_k,\riem)  \overline{\partial} (g_k - \hat{g}_k) =0 \]
	 on $\riem$ for all $k$.  Thus $J_q(\Gamma_k) (g_k - \hat{g}_k)$ is constant
	 on $\riem$.  Since $J_q(\Gamma_k) g_k$ and $J_q(\Gamma_k) \hat{g}_k$ both vanish at $q$, this proves the claim.  If $q \notin \riem$, then the claim is true up to a constant.  \end{remark}
\begin{theorem}  \label{th:average_Xr_dense}
 The set $\oplus_k \mathfrak{G}(\Omega_{k,\epsilon}, \Omega_k) X_\epsilon$ is dense in $W$.  
\end{theorem}

\begin{proof}
 The proof follows the structure of that of Theorem 4.16 in \cite{Schiffer_comparison}, but generalizes it to the case of several boundary curves.

Define $\mathcal{P} : \bigoplus_{k=1}^n \mathcal{D}(\Omega_{k,\epsilon}) \rightarrow X_\epsilon$ the orthogonal projection to the subspace $X_\epsilon$.
Fix a basis $\{\alpha_1, \dots, \alpha_g\}$  for the vector space $A(R)$, where $g$ is the genus of $R$. Define the operator $Q$ by
\begin{align*}
Q: \bigoplus_{k=1}^n \mathcal{D}(\Omega_{k, \epsilon}) &\rightarrow \mathbb{C}^g\\
(u_1, \dots, u_n)&\rightarrow  \left( \sum_{k=1}^n \int_{\Gamma_k^{\prime}} u_k \alpha_1, \dots,  \sum_{k=1}^n \int_{\Gamma_k^{\prime}} u_k \alpha_g \right)
\end{align*}
For any constants $c_k \in \mathbb{C}$, $k=1,\ldots,n$, $Q(u_1 + c_1,\ldots,u_n+c_n) = Q(u_1,\ldots,u_n)$.  Furthermore, for any fixed $k$ and fixed $x_k \in \Gamma_k'$ there is a uniform $C$ such that 
\[  \sup_{z \in \Gamma_k'} |u_k(z) - u_k(x_k) | \leq C \| u_k  - u_k(x_k) \|_{\mathcal{D}(\Omega_{k,\epsilon})}   \]
These two facts together imply that $Q$ is bounded.  

By using the Riesz representation theorem and the Gram-Schmidt process 
\begin{equation*}
\exists C>0 \hspace{.5cm} s.t.  \hspace{.5cm} \|\mathcal{P}(u_1, \dots, u_n)-(u_1, \dots, u_n)\|_{\bigoplus_{k=1}^n \mathcal{D}(\Omega_{k, \epsilon})}\leq C\|Q(u_1, \dots, u_n)\|_{\mathbb{C}^{g}}.
\end{equation*}

Now define $Q_1$ by 
\begin{equation*}
\begin{split}
Q_1: \bigoplus_{k=1}^n \mathcal{D}_{\mathrm{harm}}(\Omega_k) &\rightarrow \mathbb{C}^g\\
(h_1, \dots, h_n)&\rightarrow  \left( \lim_{s\searrow 0^+}\sum_{k=1}^n \int_{\Gamma_s^{p_k}} h_k \alpha_1, \dots,  \lim_{s\searrow 0^+}\sum_{k=1}^n \int_{\Gamma_s^{p_k}} h_k \alpha_g \right).
\end{split}
\end{equation*}
By the definition of $W$, $Q_1(h_1, \dots, h_n)=0$ if $(h_1, \dots, h_n)\in W$. By Stokes' theorem and continuity of the Dirichlet inner product, there exists $D>0$ such that 

\begin{equation*}
\|Q_1(h_1, \dots, h_n)\|_{\mathbb{C}^{g}}\leq D\|(h_1, \dots, h_n)\|_{ \bigoplus_{k=1}^n \mathcal{D}_{\mathrm{harm}}(\Omega_k)}.
\end{equation*}
By (\ref{eq:forms_same_average}) 
\begin{equation*}
Q(u_1, \dots, u_n)=Q_1(\mathfrak{G}(\Omega_{1, \epsilon}, \Omega_1)u_1, \dots, \mathfrak{G}(\Omega_{n, \epsilon}, \Omega_n)u_n)= Q_1(\oplus_k \mathfrak{G}(\Omega_{k,\epsilon}, \Omega_k)(u_1, \dots, u_n)).
\end{equation*}

Now let $(h_1, \dots, h_n)\in W$ and $\epsilon>0$, then by density of $\mathfrak{G}(\Omega_{k,\epsilon}, \Omega_k)\mathcal{D}(\Omega_{k, \epsilon})$ in $\mathcal{D}_{\mathrm{harm}}(\Omega_k)$ \cite[Theorem 4.6]{Schiffer_comparison}, for each $k=1, \ldots ,n$ there exists $u_k\in \mathcal{D}(\Omega_{k, \epsilon})$ such that 

\begin{equation*}
\|\mathfrak{G}(\Omega_{k, \epsilon}, \Omega_k)u_k-h_k\|_{\mathcal{D}_{\mathrm{harm}}(\Omega_{k})}\leq \frac{\epsilon}{\sqrt{n}},
\end{equation*}

Therefore by the Minkowski inequality we have
\begin{equation}\label{EPE}
\begin{split}
&\|\oplus_k \mathfrak{G}(\Omega_{k,\epsilon}, \Omega_k) \mathcal{P}(u_1, \dots, u_n)-(h_1, \dots, h_n)\|\\
 &\hspace{10pt}\leq \|\oplus_k \mathfrak{G}(\Omega_{k,\epsilon}, \Omega_k)\mathcal{P}(u_1, \dots, u_n)-\oplus_k \mathfrak{G}(\Omega_{k,\epsilon}, \Omega_k)(u_1, \dots, u_n)\|\\
 &\hspace{17pt}+\|\oplus_k \mathfrak{G}(\Omega_{k,\epsilon}, \Omega_k)(u_1, \dots, u_n)-(h_1, \dots, h_n)\| \\
 &\hspace{10pt}\leq\|\oplus_k \mathfrak{G}(\Omega_{k,\epsilon}, \Omega_k)\| \|\mathcal{P}(u_1, \dots, u_n)-(u_1, \dots, u_n)\|\\
&\hspace{17pt}+\|\big(\mathfrak{G}(\Omega_{1, \epsilon}, \Omega_1)u_1-h_1, \dots, \mathfrak{G}(\Omega_{n, \epsilon}, \Omega_n)u_n-h_n\big)\|
 \end{split}
\end{equation}
where all the norms are $\|.\|_{\bigoplus_{k=1}^n \mathcal{D}_{harm}(\Omega_k)}$ except the operator norm $\|\oplus_k \mathfrak{G}(\Omega_{k,\epsilon}, \Omega_k)\|$. Since $Q_1(h_1, \dots, h_n)=0$ one has

\begin{align*}
\|\mathcal{P}(u_1, \dots, u_n)-(u_1, \dots, u_n)\|&\leq C\|Q(u_1, \dots, u_n)\|_{\mathbb{C}^g}=C\|Q_1\oplus_k \mathfrak{G}(\Omega_{k,\epsilon}, \Omega_k)(u_1, \dots, u_n)\|_{\mathbb{C}^g}\\
&=C\|Q_1\big(\oplus_k \mathfrak{G}(\Omega_{k,\epsilon}, \Omega_k)(u_1, \dots, u_n)-(h_1, \dots, h_n)\big)\|_{\mathbb{C}^g}\\
&\leq CD\|\oplus_k \mathfrak{G}(\Omega_{k,\epsilon}, \Omega_k)(u_1, \dots, u_n)-(h_1, \dots, h_n)\|.
\end{align*}

By our choice of $u_k$'s for the second term we have 

\begin{align*}
\|\big(\mathfrak{G}(\Omega_{1, \epsilon}, \Omega_1)u_1-h_1, \dots, \mathfrak{G}(\Omega_{n, \epsilon}, \Omega_n)u_n-h_n\big)\|&=\left(\sum_{k=1}^n \|\mathfrak{G}(\Omega_{k, \epsilon}, \Omega_k)u_k-h_k\|^2_{\mathcal{D}_{\mathrm{harm}}(\Omega_k)}\right)^{\frac{1}{2}}\\
&\leq\left(\sum_{k=1}^n \frac{\epsilon^2}{n}\right)^{\frac{1}{2}}=\epsilon.
\end{align*}

Combining the above two inequalities with (\ref{EPE}) 
\begin{align*}
&\|\oplus_k \mathfrak{G}(\Omega_{k,\epsilon}, \Omega_k)\mathcal{P}(u_1, \dots, u_n)-(h_1, \dots, h_n)\|\\
&\leq CD\|\oplus_k \mathfrak{G}(\Omega_{k,\epsilon}, \Omega_k)\|\,\|\oplus_k \mathfrak{G}(\Omega_{k,\epsilon}, \Omega_k)(u_1, \dots, u_n)-(h_1, \dots, h_n)\|+\epsilon \\
&\leq CD\|\oplus_k \mathfrak{G}(\Omega_{k,\epsilon}, \Omega_k)\| \epsilon+\epsilon=(CD\|\oplus_k \mathfrak{G}(\Omega_{k,\epsilon}, \Omega_k)\|+1)\epsilon.\\
\end{align*}

Therefore  $\oplus_k \mathfrak{G}(\Omega_{k,\epsilon}, \Omega_k)X_\epsilon$ is dense in $W$. 

\end{proof}

\end{subsection}
\begin{subsection}{Proof of Theorem \ref{th:general_T_is_isomorphism}}
  The proof proceeds in several steps.  

 We will show that the Schiffer operator is a bounded isomorphism on the subspace $V$ of $\bigoplus_{k=1}^n \overline{A(\Omega_k)}$. Recall that $V$ is given by 
\begin{equation*}
V=\Big\{ (\overline{\alpha}_1, \dots, \overline{\alpha}_n)\in \bigoplus_{k=1}^n \overline{A(\Omega_k)} : \sum_{k=1}^n \iint_{\Omega_k} \beta \wedge \overline{\alpha}_k=0~;~ \forall {\beta} \in {A(R)} \Big\}
\end{equation*}
where we have made use of the isomorphism (\ref{eq:obvious_isomorphism}) to rewrite $V$ in terms of the restrictions to $\Omega_k$.  We record the following obvious fact.
\begin{lemma}\label{VBC}
The operator 
\begin{align*}
\overline{\mathbf{\partial}} : W' &\rightarrow V  \\
(\overline{h}_1, \dots, \overline{h}_n)&\rightarrow (\overline{\partial h}_1, \dots, \overline{\partial h}_n),
\end{align*}
is a surjective operator which preserves the norm.
\end{lemma}

For fixed $j=1,\ldots,n$ we will define the transmission operator $\mathfrak{O}(\riem, \Omega_j)$ from $\riem$ to $\Omega_j$.  Recall that $\Omega_j^*$ is the complement of the closure of $\Omega_j$ in $R$, which contains $\riem$.    For fixed $p'_j \in \Omega_j^*$, let $\Omega^*_{j,p'_j,\epsilon}$ be a doubly-connected domain bounded by $\Gamma_j$ in $\riem$ as in (\ref{eq:canonical_neighb_riem}).  Let $Res(\riem, \Omega_{j,p_j',\epsilon}^*)$ be the restriction operator from $\mathcal{D}(\riem)$ to $\mathcal{D}(\Omega_{j,p_j',\epsilon}^*)$.  
We then define
\begin{definition}\label{TMC} 
$\mathfrak{O}(\riem, \Omega_j) := \mathfrak{O}(\Omega^*_j,\Omega_j) \mathfrak{G}(\Omega_{j,p_j',\epsilon}^*,\Omega^*_j) Res(\riem, \Omega_{j,p_j',\epsilon}^*).$
\end{definition}
The interpretation of $\mathfrak{O}(\riem,\Omega_j)$ is that it takes elements of $\mathcal{D}_{\text{harm}}(\riem)$ to elements of $\mathcal{D}_{\text{harm}}(\Omega_j)$ 
with the same CNT boundary values on $\Gamma_j$; in other words, it is a transmission operator from $\riem$ to $\Omega_j$.  This is independent of the choice of $p_j'$ \cite{Schippers_Staubach_general_transmission}. The above expression establishes that the operator is bounded, since each operator on the right hand side is bounded.
Finally, let 
\begin{align*}
 \mathfrak{O}(\riem,\mathcal{O}): \mathcal{D}_{\text{harm}}(\riem) & \rightarrow \bigoplus_{k=1}^n \mathcal{D}_{\text{harm}}(\Omega_k) \\
 h & \mapsto \left( \mathfrak{O}(\riem,\Omega_1)h,\ldots,\mathfrak{O}(\riem,\Omega_n)h \right).
\end{align*}

\begin{theorem}  \label{th:transmitted_jump} 
 For  all $\overline{h} = (\overline{h}_1,\ldots, \overline{h}_n) \in W'$, we have 
 \[  -\mathfrak{O}(\riem,\Omega_j) [J_q(\Gamma)\overline{h}]_{\riem} = \overline{h}_j - [J_q(\Gamma) \overline{h}]_{\Omega_j}  \]
 for $j=1,\ldots,n$.  That is
 \[    -\mathfrak{O}(\riem,\mathcal{O}) [J_q(\Gamma)\overline{h}]_{\riem} = \overline{h} - [J_q(\Gamma) \overline{h}]_{\mathcal{O}}.     \]
\end{theorem}
\begin{proof}
 Since every operator in the identity above is bounded, it suffices to prove this for $\oplus_k \mathfrak{G}(\Omega_{k,\epsilon}, \Omega_k) X_\epsilon$, because this set is dense by Theorem \ref{th:average_Xr_dense}.  
 
For each $k=1, 2, \dots, n$, let $\Omega_{k, \epsilon}$ be as above in $\Omega_k$, and let $h_k\in \mathcal{D}(\Omega_{k, \epsilon})$. 
We will apply \cite[Theorem 4]{Royden} with  $\mathcal{E}=\cup_{k=1}^n(\Omega_k\backslash \Omega_{k, \epsilon})$ and $\mathcal{O}=\cup_{k=1}^n\Omega_k$ which are a closed subset and an open subset of $R$, respectively. We also have $\mathcal{O}\backslash \mathcal{E}=\cup_{k=1}^n \Omega_{k, \epsilon}$.

Let $h= (h_1, \dots, h_n)\in \bigoplus_{k=1}^n \mathcal{D}(\Omega_{k, \epsilon})$ satisfying the integral condition 
\begin{equation*}
\sum_{k=1}^n \int_{\Gamma_k^\prime} h_k \alpha=0~~~; ~~~\forall \alpha \in A(R),
\end{equation*}
where for each $k$, $\Gamma_k^\prime$ is an analytic curve in $\Omega_{k, \epsilon}$ isotopic to $\Gamma_k$ in the closure of $\Omega_{k, \epsilon}$.

By \cite[Theorem 4]{Royden}, there exists $H_1\in \mathcal{D}(\mathcal{O})$ ($F-f$ in \cite{Royden}) and $H_2\in \mathcal{D}( R\backslash \mathcal{E})=\mathcal{D}(\mathrm{cl}(\riem)\cup (\cup_{k=1}^n\Omega_{k, \epsilon}))$ ($-f$ in \cite{Royden}) which satisfy  
\begin{equation}  \label{eq:Royden_temp}
h(z)=H_1(z)-H_2(z)~~~; ~~~~ \forall z\in  \mathcal{O}\backslash \mathcal{E}=\cup_{k=1}^n \Omega_{k, \epsilon}.
\end{equation}
These are explicitly given by the following formulas.  Setting  $\Gamma^\prime=\cup_{k=1}^n \Gamma_k^\prime$ for $z\in \Omega_j(\subset \mathcal{O})$, we have 
\begin{align*}
H_1|_{\Omega_j}(z)=[J_q(\Gamma^\prime)'h]_{\mathcal{O}} (z)
&=-\frac{1}{\pi i}\sum_{k=1}^n \int_{\Gamma_k^\prime}\partial_w g(w; z, q)\,h_k(w)\\
&=\sum_{k\neq j}[J_q(\Gamma_k^\prime)'h_k]_{\Omega_j}(z)+ [J_q(\Gamma_j^\prime)'h_j]_{\Omega_j}(z)\\
&=\sum_{k\neq j}[J_q(\Gamma_k)_{\Omega^*_k} \mathfrak{G}(\Omega_{k, \epsilon}, \Omega_k)h_k]_{\Omega_j}(z)+ [J_q(\Gamma_j)\mathfrak{G}(\Omega_{j, \epsilon}, \Omega_j)h_j]_{\Omega_j}(z)\\
\end{align*}
where the last equality stems from \cite[Theorem 4.10]{Schiffer_comparison}. Applying \cite[Theorem 4.10]{Schiffer_comparison} again, we have for $z\in \riem$ 
\begin{equation*}
\begin{split}
H_2|_{\riem}(z) 
&=-\frac{1}{\pi i}\sum_{k=1}^n \int_{\Gamma_k^\prime}\partial_w g(w; z, q)\,h_k(w)\\
&=\sum_{k=1}^n [J_q(\Gamma_k^\prime)'h_k]_{\riem}(z)\\
&=\sum_{k=1}^n [J_q(\Gamma_k)_{\Omega^*_k}\mathfrak{G}(\Omega_{k, \epsilon}, \Omega_k)h_k]_{\riem}(z).\\
\end{split}
\end{equation*}

Restricting (\ref{eq:Royden_temp}) to $\Omega_{j,\epsilon}$ yields  
\begin{equation} \label{eq:Royden_temp_two}  h_j(z)=h|_{\Omega_{j, \epsilon}}(z)=H_1|_{\Omega_{j, \epsilon}}(z)-H_2|_{\Omega_{j, \epsilon}}(z)~~~; ~~ \forall z\in  \Omega_j.   
\end{equation}
Now since we have
\begin{align*}
 \mathfrak{G}(\Omega_{j,\epsilon},\Omega_j) \left( \left. H_1 \right|_{\Omega_{j, \epsilon}} \right) & =  H_1 ,\\
 \mathfrak{G}(\Omega_{j,\epsilon},\Omega_j) \left( \left. H_2 \right|_{\Omega_{j, \epsilon}} \right)& = \mathfrak{O}(\riem,\Omega_j)\left(\left. H_2 \right|_\riem \right), 
\end{align*}
applying $\mathfrak{G}(\Omega_{j, \epsilon},\Omega_j)$ to each term of (\ref{eq:Royden_temp_two}) we obtain
\[   \mathfrak{G}(\Omega_{j, \epsilon}, \Omega_j)h_j(z)=H_1|_{\Omega_j}(z)-\mathfrak{O}(\riem, \Omega_j)(H_2|_\riem)(z)~~;~~ \forall z\in  \Omega_j.   \]

Finally, inserting the formulas for $H_1$ and $H_2$ yields that
\begin{align*}
 \mathfrak{G}(\Omega_{j, \epsilon}, \Omega_j)h_j(z)=&\sum_{k\neq j}[J_q(\Gamma_k)_{\Omega^*_k}\mathfrak{G}(\Omega_{k, \epsilon}, \Omega_k)h_k]_{\Omega_j}(z)+ [J_q(\Gamma_j)\mathfrak{G}(\Omega_{j, \epsilon}, \Omega_j)h_j]_{\Omega_j}(z)\\
&-\mathfrak{O}(\riem, \Omega_j)\sum_{k=1}^n [J_q(\Gamma_k)_{\Omega^*_k}\mathfrak{G}(\Omega_{k, \epsilon}, \Omega_k)h_k]_{\riem}(z)~~~; ~~\forall z\in \Omega_j
\end{align*}
which completes the proof.
\end{proof}

\begin{remark}
 It is easily seen that this holds trivially for all holomorphic $h \in W$, since the left hand side vanishes by Remark \ref{re:ignore_holomorphic}.   Thus the theorem holds for all $h \in W$. 
\end{remark}

\begin{lemma}\label{NBV}
If $h\in  \mathcal{D}(\riem)$ then
\begin{equation*}
h = J_q(\Gamma)_{\riem} \mathfrak{O}(\riem,\mathcal{O}) h.
\end{equation*}
\end{lemma}
\begin{proof}   We will distinguish $J$ defined by limiting integrals from within $\Omega_k$ and from within $\Omega_k^*$ in this proof.
 Recall that
 \[  J_q(\Gamma)_{\riem} \mathfrak{O}(\riem,\mathcal{O}) h = \sum_{k=1}^n [J_q(\Gamma_k)\mathfrak{O}(\riem, \Omega_k)h]_\riem.\]
 
Since the function $h$ is holomorphic on $\riem$ it is equal to the sum of the limiting integrals from within $\Omega_k^*$ for the boundary curves $\Gamma_k$. i.e.

\begin{align*}
h(z)=-\lim_{s\searrow 0}\frac{1}{\pi i}\sum_{k=1}^n \int_{\Gamma_s^k} \partial_w g_R (w; z, q) h(w)\,dw.\\  
\end{align*}

Let $\Omega_{k, \epsilon}^*$ be the doubly-connected domain bounded by $\Gamma_k$ in $\riem$. One can replace each integral in the sum by an integral over a fixed analytic curve $\Gamma'_k$ in $\Omega^*_{k,  \epsilon}$. That defines an operator $J_q(\Gamma_k, \Omega^*_{k, \epsilon})'$. For every $k=1,\ldots,n$,  \cite[Theorem 4.9]{Schiffer_comparison} yields that
$$J_q(\Gamma_k, \Omega^*_{k, \epsilon})'_\riem h = J_q(\Gamma_k, \Omega^*_k)_\riem  [\mathfrak{G}(\Omega_{k, \epsilon}^*,\Omega^*_k)(h|_{\Omega^*_{k, \epsilon}})]. $$

Now apply \cite[Theorem 4.10]{Schiffer_comparison} for each fixed curve $\Gamma_k$ and function $\mathfrak{G}(\Omega^*_{k,\epsilon},\Omega_k^*) \left(\left. h^* \right|_{\Omega_{k,\epsilon}} \right)$ to obtain that  
$$J_q(\Gamma_k, \Omega^*_k)_\riem  [\mathfrak{G}(\Omega_{k, \epsilon}^*,\Omega^*_k)(h|_{\Omega^*_{k, \epsilon}})] = J_q(\Gamma_k, \Omega_k)_\riem \,\mathfrak{O}(\Omega^*_k,\Omega_k)  [\mathfrak{G}(\Omega_{k, \epsilon}^*,\Omega^*_k)(h|_{\Omega^*_{k, \epsilon}})].$$

Finally, taking a sum over all terms and using definition \ref{TMC} yield that
\begin{equation*}
\begin{split}
h&=\sum_{k=1}^n [J_q(\Gamma_k, \Omega^*_{k, \epsilon})'_\riem (h|_{\Omega^*_{k, \epsilon}})]\\
&=\sum_{k=1}^n J_q(\Gamma_k, \Omega_k)_\riem \mathfrak{O}(\Omega^*_k,\Omega_k) [\mathfrak{G}(\Omega_{k, \epsilon}^*,\Omega^*_k)(h|_{\Omega_{k, \epsilon}^*})]\\
&=\sum_{k=1}^n J_q(\Gamma_k, \Omega_k)_\riem [\mathfrak{O}(\riem, \Omega_k)h],
\end{split}
\end{equation*}

which completes the proof. 
\end{proof}


\begin{theorem}\label{TIS}
$T(\mathcal{O},  \riem)$ is a surjective operator from $V$ onto $A(\riem)_e$.
\end{theorem}

\begin{proof}
First we show that $T(\Omega_1, \dots , \Omega_n,  \riem)(V)\subset A(\riem)_e$.

Let $\overline{\alpha}=(\overline{\alpha}_1, \dots, \overline{\alpha}_n)\in V$.   By Lemma \ref{VBC} there is an $\overline{H} = ( \overline{H}_1,\ldots,\overline{H}_n) \in W'$ such that $\overline{\alpha} = \overline{\partial} \overline{H}$. Moreover, Lemma \ref{CDJO} and Remark \ref{re:ignore_holomorphic} yield that $J_q(\Gamma) (\overline{H}_1,\ldots,\overline{H}_n)$ is holomorphic.  Therefore for $z \in \riem$,  Theorem \ref{DJO}  yields
\begin{align*}
 T(\mathcal{O},\riem)(\overline{\alpha}_1,\ldots,\overline{\alpha}_n) & = 
 \partial J_q(\Gamma)_\riem  (\overline{H}_1,\ldots,\overline{H}_n)  \\ & = 
 d J_q(\Gamma)_\riem (\overline{H}_1,\ldots,\overline{H}_n) \in A(\riem)_e.
\end{align*}

Next we show that every element in $A(\riem)_e$ is in the image of $T(\mathcal{O},\riem)$. 
Given $\beta\in A(\riem)_e$, then there exists an $h_{\riem} \in  \mathcal{D}(\riem)_q$ such that $\partial_z h_{\riem} =\beta$. Let $h_k\in \mathcal{D}_{\text{harm}}(\Omega_k)$ be such that $\mathfrak{O}( \riem, \Omega_k)h_{\riem} =h_k$, i.e. $h_{\riem}$ and $h_k$ have the same CNT boundary values on $\Gamma_k$.
Lemma \ref{NBV} and Theorem \ref{DJO} now imply that 
\begin{align*}
\beta&=\partial h_{\riem}=\partial\sum_{k=1}^n J_q(\Gamma_k, \Omega_k)_ \riem[\mathfrak{O}( \riem, \Omega_k)h_{\riem}]\\
&=\partial J_q(\Gamma)_\riem (h_1,\ldots,h_n)  \\
 & = T(\mathcal{O},\riem)(\overline{\partial} h_1,\ldots,\overline{\partial} h_n ).  
\end{align*}

So we need only show that $( \overline{\partial} h_1, \dots,  \overline{\partial} h_n)$ is in $V$; that is, for all $\overline{\alpha}\in \overline{A(R)}$, 
\[ \sum_{k=1}^n \iint_{\Omega_k, w} \alpha \wedge \overline{\partial}h_k=0.  \]
To see this we have
\begin{align} \label{eq:surjective_temp}
\sum_{k=1}^n \iint_{\Omega_k, w} \alpha \wedge \overline{\partial}h_k &= \sum_{k=1}^n \iint_{\Omega_k, w} \alpha(w)\wedge  \overline{\partial}h_k(w)\nonumber \\
&=\sum_{k=1}^n \iint_{\Omega_k, w} \left( \iint_{R, z} K_ R(w, z)\wedge_z  \alpha(z)\right) \wedge_w \overline{\partial}h_k(w) \nonumber \\
&=\sum_{k=1}^n  \iint_{ R, z}  \alpha(z) \wedge_z \left( \iint_{\Omega_k, w} \overline{K_ R(z, w)} \wedge_w \overline{\partial}h_k(w) \right) \nonumber \\
&=\sum_{k=1}^n  \iint_{ R, z}  \alpha(z) \wedge_z \left( \iint_{\Omega_k, w}  \overline{\partial}_z \partial_w g(w; z, q) \wedge_w \overline{\partial}h_k(w) \right) \nonumber \\
&=\iint_{ R, z}  \alpha(z) \wedge_z \left(\sum_{k=1}^n \iint_{\Omega_k, w}  \overline{\partial}_z \partial_w g(w; z, q) \wedge_w \overline{\partial}h_k(w) \right).
\end{align}

Using Lemma \ref{NBV} once again, we have 
\[ h_{\riem} =-\frac{1}{\pi i}\sum_{k=1}^n \iint_{\Omega_k, w} \partial_w g(w; z, q) \wedge_w \overline{\partial}_w h_k(w).  \]  
On the other hand, $h_{\riem}$ is holomorphic so $ \overline{\partial}_zh_{\riem}=0$.  Therefore 
$$\sum_{k=1}^n  \iint_{\Omega_k,w}  \overline{\partial}_z \partial_w g(w; z, q) \wedge_w \overline{\partial}h_k(w)=0$$
which inserted in (\ref{eq:surjective_temp}) completes the proof. 
\end{proof}

We note that the transmission operator $\mathfrak{O}(\riem,\Omega)$ induces a transmission on the set of exact forms by conjugating by differentiation.  Namely,  for fixed $k$ set
\begin{align}
 \mathfrak{O}_e(\riem,\Omega_k) & = d \mathfrak{O}(\riem,\Omega_k) d^{-1}:A(\riem)_e \rightarrow A(\Omega_k)_e \nonumber \\
 \mathfrak{O}_e(\riem,\mathcal{O}) & = d \mathfrak{O}(\riem,\mathcal{O}) d^{-1}:A(\riem)_e \rightarrow A(\mathcal{O})_e.
\end{align}
Although $d^{-1}$ is not well-defined because of the arbitrary choice of constant, $\mathfrak{O}_e(\riem,\Omega_k)$ is well-defined because $\mathfrak{O}(\riem,\Omega_k)$ takes constants to constants.

Defining 
\begin{align*}
 T(\mathcal{O},\mathcal{O}):\overline{A(\mathcal{O})} & \rightarrow A(\mathcal{O}) \\
 \alpha & \mapsto \frac{1}{\pi i} \iint_{\mathcal{O}} \partial_z \partial_w g_R(w;z,q) \wedge_w \overline{\alpha(w)} \ \ z \in \mathcal{O} 
\end{align*}
we then have the following version of Theorem \ref{th:transmitted_jump} for one-forms.

\begin{theorem}  \label{th:transmitted_jump_forms} For all $\overline{\alpha}= (\overline{\alpha}_1,\ldots,\overline{\alpha}_n) \in V$, 
 \[  \mathfrak{O}_e(\riem,\Omega_j) T(\mathcal{O},\riem) \overline{\alpha} = \overline{\alpha}_j  
 + T(\mathcal{O},\Omega_j) \overline{\alpha} \]
 for $j=1,\ldots,n$.  That is
 \[  \mathfrak{O}_e(\riem,\mathcal{O}) T(\mathcal{O},\riem) \overline{\alpha} = \overline{\alpha}   
 + T(\mathcal{O},\mathcal{O}) \overline{\alpha}. \]
\end{theorem}
\begin{proof}
 This follows from Theorems \ref{DJO} and \ref{th:transmitted_jump}.
\end{proof}

Now for any open set $D$ of $R$  let $P_A(D):A_{\text{harm}}(D) \rightarrow A(D)$ and $\overline{P}_A(D):A_{\text{harm}}(D) \rightarrow \overline{A(D)}$ denote the orthogonal projections.  

\begin{corollary}\label{TII}
$T(\mathcal{O}, \riem)$ is injective on $V$, with left inverse $\overline{P}_A(\mathcal{O}) \mathfrak{O}_e(\riem,\mathcal{O})$.
\end{corollary}
\begin{proof}
 Apply $\overline{P}_A(\mathcal{O})$ to both sides of the second equation of Theorem \ref{th:transmitted_jump_forms}.  
\end{proof}

Now observe that Theorem \ref{th:general_T_is_isomorphism} follows directly from Theorem \ref{TIS} and Corollary \ref{TII}.  
\end{subsection}
\end{section}
\begin{section}{{Applications of the isomorphism theorem}}
\begin{subsection}{Plemelj-Sokhtoski jump problem for finitely many quasicircles}  \label{se:jump}
 In this section we establish a jump formula for $n$ quasicircles.  Setting aside analytic issues momentarily, the problem is as follows.  Given a function $U$ on $\Gamma = \Gamma_1 \cup \cdots \cup \Gamma_n$, find holomorphic functions $u_k$ on $\Omega_k$ and $u_\riem$ on $\riem$ such that on each curve $\Gamma_k$ the boundary values $\tilde{u}_k$ and $\tilde{u}_{\riem}$ respectively satisfy
 \[  \tilde{u}_k - \tilde{u}_{\riem} = u.  \]
 The solution to this problem is well-known for more regular curves, say for $\Gamma$ and $u$ smooth.  
 Here, $\Gamma_k$ are of course quasicircles.  We consider the class of functions $\mathcal{H}(\Gamma_k)$; recall that these functions are CNT boundary values of elements of $\mathcal{D}_{\text{harm}}(\Omega_k)$.   
 
It is classically known \cite{Gakhov_book}, \cite{Rodin_book} that the topological condition for existence of a solution to the jump problem for functions $U$ on $\Gamma$ is that 
 \[  \sum_{k=1}^n \int_{\Gamma}  \alpha U =0   \]
 for all one-forms $\alpha \in A(R)$.  On quasicircles, this integral condition would not make sense, because quasicircles need not be rectifiable.  Thus, we replace this by the condition that $U$ is the boundary values of an element of $W$, motivated by (\ref{eq:jump_condition_doubleintegral}).    
 
 Our first theorem in some sense is the derivative of the jump isomorphism.
 Let
 \[  V' = \{ \overline{\alpha} + \beta \in A_{\text{harm}}(\mathcal{O}) : 
  \overline{\alpha} \in V  \ \ \text{and} \ \ \beta \in A(\mathcal{O}) \}.  \]
 \begin{theorem}  \label{th:derivative_of_jump_isomorphism}
  Let $R$ be a compact Riemann surface, and $\Omega_1,\ldots,\Omega_n$ be simply connected regions in $R$, bounded by quasicircles $\Gamma_1,\ldots,\Gamma_n$.  Assume that the closures of $\Omega_1,\ldots,\Omega_n$ are pairwise disjoint.    Then 
  {\begin{align*}  
   \hat{\mathfrak{H}}:V' & \rightarrow A(\mathcal{O}) \oplus A(\riem)_e \\
    \overline{\alpha} + \beta & \mapsto \left( \beta -  T(\mathcal{O},\mathcal{O})  \overline{\alpha} ,
    -T(\mathcal{O},\riem) \overline{\alpha} \right) 
  \end{align*}}
  is an isomorphism.  
 \end{theorem}
 \begin{proof}
  Note that $\hat{\mathfrak{H}}$ is well-defined, since the output is independent of the choice of constant in $h$.  
  First we show that it is injective.  Assume that { $\hat{\mathfrak{H}}( \overline{\alpha} + \beta) = 0$,  then $\overline{\alpha} =0$ by Theorem \ref{th:general_T_is_isomorphism}.  
  But since $0 = \beta - T(\mathcal{O},\mathcal{O}) \overline{\alpha} = \beta,$ we also have that $\beta = 0$.  }
  
  Now we show that $\hat{\mathfrak{H}}$ is surjective.  Let $(\beta_{\mathcal{O}},\beta_{\riem}) \in A(\mathcal{O}) \oplus A(\riem)_e$.  
  By Theorem \ref{th:general_T_is_isomorphism} there is an $\overline{\alpha} \in V$ such that $-T(\mathcal{O},\riem)\overline{\alpha} = \beta_{\riem}$.  Setting {$\beta = \beta_{\mathcal{O}} + T(\mathcal{O},\mathcal{O}) \overline{\alpha}$ yields that $\hat{\mathfrak{H}}( \beta + \overline{\alpha}) = (\beta_\mathcal{O},\beta_{\riem})$.  }
 \end{proof}
 
 \begin{theorem}  \label{th:jump_isomorphism}
  Let $R$, $\mathcal{O}$, and $\riem$ be as in \emph{Theorem \ref{th:derivative_of_jump_isomorphism}}.  Then 
   \begin{align*}
 {\mathfrak{H}}: \mathcal{D}_{\mathrm{harm}}(\mathcal{O}) & \rightarrow \mathcal{D}(\mathcal{O}) \oplus \mathcal{D}(\riem)_q \\
  h &\mapsto \left(\left. J_q(\Gamma) h \right|_{\mathcal{O}},\left. J_q(\Gamma) h \right|_{\riem} \right)   
 \end{align*}
 is a bounded isomorphism from $W$ to $\mathcal{D}(\mathcal{O}) \oplus \mathcal{D}(\riem)_q$.  
 \end{theorem}
 \begin{proof}
  By Lemma \ref{CDJO} the image of $\mathfrak{H}$ is in $\mathcal{D}(\mathcal{O}) \oplus \mathcal{D}(\riem)$.  Since $g(w;z,q)$ vanishes identically at $z=q$ (see \cite{Schiffer_comparison} proof of Theorem 4.26), so does $\partial_w g(w;z,q)$.  Thus the image of $\mathfrak{H}$ is in $\mathcal{D}(\mathcal{O}) \oplus \mathcal{D}(\riem)_q$.  
  
  By Theorem \ref{DJO}, $\partial \mathfrak{H} = \hat{\mathfrak{H}} d$.  Assume that $\mathfrak{H} h =0$.  Then $\hat{\mathfrak{H}} dh =0$, so $dh = 0$, so $h$ is constant.   But if $h$ is a constant $c$ then $\left. J_q(\Gamma) c \right|_{\riem}= c$.  Since $c \in \mathcal{D}(\Gamma)_q$ it vanishes at $q$, so $c=0$.  So $\mathfrak{H}$ is injective.  
  
  We also have that $\mathfrak{H} (h + c) = \mathfrak{H} h + (c,0)$ for any constant $c$.  This together with the fact that $\hat{\mathfrak{H}}$ is surjective shows that $\mathfrak{H}$ is surjective.
 \end{proof}
 
 The proof of Theorem \ref{th:jump_isomorphism} also shows the following.
 \begin{theorem}  \label{th:jump_isomorphism_just_one_side}
 Let $R$, $\mathcal{O}$, and $\riem$ be as in \emph{Theorem \ref{th:derivative_of_jump_isomorphism}}. 
 Then the restriction of $J_q(\Gamma)_{\riem}$ to $W$ is an isomorphism onto $\mathcal{D}(\riem)_q$.  
 \end{theorem}
 
 \begin{theorem} \label{th:jump_dependable} Let $R$, $\mathcal{O}$, and $\riem$ be as in \emph{Theorem \ref{th:derivative_of_jump_isomorphism}}.  Let $q \in \riem$.
 If $h \in \mathcal{D}(\mathcal{O})$ then $J_q(\Gamma) h = (h,0)$, and if $h \in \mathcal{D}(\riem)_q$
 then $J_q(\Gamma) \mathfrak{O}(\riem,\mathcal{O}) h = (0,-h)$.  
 \end{theorem}
 \begin{proof}  If $h \in \mathcal{D}(\mathcal{O})$, then since $\partial_w g(w;z,q)$ is holomorphic except for a simple pole of residue one at $w=z$, by the residue theorem $J_q(\Gamma) h = h$.  If $h \in \mathcal{D}(\riem)_q$ then similarly $J_q(\Gamma,\riem) h = - h + h(q) = - h$  (note that $\Gamma$ is negatively oriented with 
  respect to $\riem$).   By (\ref{eq:J_both_sides}) and (\ref{eq:TOsig_sum}),   $J_q(\Gamma)h = J_q(\Gamma,\riem) h$, which completes the proof.  
 \end{proof}

 We now prove Theorem \ref{th:jump_proper}.

 \begin{proof}(of Theorem \ref{th:jump_proper})
 Let $h \in \mathcal{D}_{\text{harm}}(\mathcal{O}) = (h_1,\ldots,h_n)$, and define
 \[ (h_{\mathcal{O}},h_{\riem}) = {\mathfrak{H}}h. \]
 Theorem \ref{th:jump_dependable} yields
  \[ {\mathfrak{H}} (- \mathfrak{O}(\riem,\mathcal{O}) h_{\riem} + h_{\mathcal{O}}) = (h_{\mathcal{O}},h_{\riem}) = {\mathfrak{H}}h. \]
  Thus by Theorem \ref{th:jump_isomorphism}
 \[  h = - \mathfrak{O}(\riem,\mathcal{O}) h_{\riem} + h_{\mathcal{O}},  \] 
 so $H = -H_{\riem} + H_k$ on $\Gamma_k$ for $k=1,\ldots,n$.

  We need only show that the solution is unique.  Given any other solution $(u_{\mathcal{O}},u_{\riem})$ 
  we have that $-\mathfrak{O}(\riem,\mathcal{O})(u_{\riem} - h_{\riem}) + (u_{\mathcal{O}} - h_{\mathcal{O}}) \in \mathcal{D}_{\text{harm}}(\mathcal{O})$ has boundary 
  values zero, so by uniqueness of the extension it is zero.  Thus 
  \[   0= {\mathfrak{H}}\left( -\mathfrak{O}(\riem,\mathcal{O})(u_{\riem} - h_{\riem}) + (u_{\mathcal{O}} - h_{\mathcal{O}}) \right)
    = (u_{\mathcal{O}}- h_{\mathcal{O}},u_{\riem} - h_{\riem}) \]
  which proves the claim.
 \end{proof}
\end{subsection}
\begin{subsection}{The approximation theorems}
In this section, we prove some approximation theorems for Dirichlet and Bergman spaces of nested Riemann surfaces, including Theorem \ref{th:dirichlet_density_squeeze} and Corollary \ref{co:embedding_in_double}.

Since the Dirichlet semi-norm is not a norm, the meaning of density requires a clarification.  Below, whenever we say that a linear subspace $Y$ of a Dirichlet space is dense in a Dirichlet space, the space $Y$ contains all constant functions.   Thus, when we approximate in the Dirichlet semi-norm, we are still free to adjust any ``approximating'' function by a constant without leaving $Y$. 

Let
\[  \overline{P}_{\Omega_k}: \mathcal{D}_{\text{harm}}(\Omega_k) \rightarrow
	 \overline{\mathcal{D}(\Omega_k)}  \]
denote orthogonal projection, where it is understood that for constants $c$
$\overline{P}_{\Omega_k} c = c$.  Let
\[  \oplus_k \overline{P}_{\Omega_k}: \oplus_k \mathcal{D}_{\text{harm}}(\Omega_k) \rightarrow
 	\oplus_k \overline{\mathcal{D}(\Omega_k)}  \]
be the direct sum of these operators.

\begin{corollary} \label{co:GX_r_antiholo_is_dense}
 The projection of $\oplus_k \mathfrak{G}(\Omega_{k,\epsilon}, \Omega_k) X_\epsilon$ onto the anti-holomorphic parts is dense in $W'$.  
\end{corollary}
\begin{proof}
	It is easily verified that the projection $\oplus_k \overline{P}_{\Omega_k}$  
	takes $W$ into $W'$.  Since it is a bounded surjective operator, the claim follows immediately from Theorem \ref{th:average_Xr_dense}.   
\end{proof}

This leads to the following density theorem.  
\begin{theorem}  \label{th:central_density_theorem}
	The restrictions of functions in $\mathcal{D}(\mathrm{cl} \, \riem \cup \Omega_{1,\epsilon} \cup \cdots \cup \Omega_{n,\epsilon})$ to $\riem$ are dense in $\mathcal{D}(\riem)$.  
\end{theorem}	
\begin{proof}
 Let $q \in \riem$.  
 By Remark \ref{re:ignore_holomorphic}, the image of $W$ under $J_q(\Gamma)_{\riem}$ is equal to the image of $W'$ under $J_q(\Gamma)_{\riem}$.  By Theorem \ref{th:jump_isomorphism_just_one_side}
 $J_q(\Gamma)$ is a bounded isomorphism from $W'$ to $\mathcal{D}(\riem)_q$.  
 
Thus by Corollary \ref{co:GX_r_antiholo_is_dense},
 \begin{equation} \label{eq:set_temp}
   J_q(\Gamma)_{\riem} \oplus_k \overline{P}_{\Omega_k} (\oplus \mathfrak{G}(\Omega_{k,\epsilon},\Omega_k) X_\epsilon) 
 \end{equation}
 is dense in $\mathcal{D}(\riem)_q$. Now Corollary \ref{co:holo_extension} and Remark \ref{re:ignore_holomorphic}  yield that all of the functions in the set (\ref{eq:set_temp}) have holomorphic extensions to $\text{cl} \, \riem \cup \Omega_{1,\epsilon} \cup \cdots \cup \Omega_{n,\epsilon}$.  Since constant functions automatically have such extensions, this completes the proof.  
\end{proof}
 
\begin{corollary} \label{co:density_general_isotopy}
 Let $R$ be a compact Riemann surface and $\riem \subset R$ be a Riemann surface such that the inclusion map is holomorphic and the boundary of $\riem$ consists of a finite number of pair-wise disjoint quasicircles $\Gamma_1, \ldots, \Gamma_n$ in $R$.  
 
 Assume that there is an open set $\riem' \subset R$ which contains $\riem$, and is bounded by quasicircles $\Gamma_k'$, $k=1,\ldots,n$, which are isotopic in the closure of $\riem' \backslash \riem$ to $\Gamma_k$ for $k=1,\ldots,n$ respectively.  Then the set of restrictions of elements of $\mathcal{D}(\riem')$ to $\riem$ is dense in $\mathcal{D}(\riem)$.
\end{corollary} 
\begin{proof}
 Consider the compact Riemann surface $R'$ obtained from $\riem'$ by sewing disks $\disk$ to the quasicircles $\Gamma_k'$ for $k=1,\ldots,n$ using fixed quasisymmetric parametrizations $\tau_k:\mathbb{S}^1 \rightarrow \Gamma_k'$, $k=1,\ldots,n$, say.  It was shown in \cite{RS_monster} that the topological space obtained from such a sewing  has a unique complex structure compatible with that of $\riem'$ and the sewn disks.  Let $\Omega_1,\ldots, \Omega_n$ be the connected components of the complement of $\riem$ in $R'$ containing $\Gamma_1',\ldots,\Gamma_n'$ respectively. It follows from the hypotheses that each $\Omega_k$ is conformally equivalent to a disk bordered by $\Gamma_k$. 
 	
 For each $k=1,\ldots,n,$ fix a point $p_k \in \Omega_k \backslash \text{cl} \riem'$, and let $f_k:\disk \rightarrow \Omega_k$ be conformal maps such that $f(0)=  p_k$.   We claim that for some $\epsilon >0$, $\riem'$ is contained in $\text{cl} \riem \cup \Omega_{1,\epsilon} \cup \cdots \cup \Omega_{n,\epsilon}$.   To see this, observe that the  set $f_k^{-1}(\Gamma_k')$ is compact and does not contain $0$.  Thus 
 \[    R_k = \inf_{p \in f_k^{-1}(\Gamma_k')}\{|p|\} >0.   \]
 Setting $r =\text{min} \{ R_1,\ldots, R_n \}/2$ and $\epsilon = - \log{r}$ proves the claim.
 
 Applying Theorem \ref{th:central_density_theorem} we obtain that $\mathcal{D}(\text{cl} \riem \cup \Omega_{1,\epsilon} \cup \cdots \cup \Omega_{n,\epsilon})$ is dense in $\mathcal{D}(\riem)$.  Since $\mathcal{D}(\riem')$ contains the restrictions of elements of $\mathcal{D}(\text{cl} \riem \cup \Omega_{1,\epsilon} \cup \cdots \cup \Omega_{n,\epsilon})$ to $\riem'$, this completes the proof.
\end{proof}

We now address the case of one-forms.
\begin{theorem}  \label{th:form_general_isotopy_density}
 Let $R$, $\riem$, and $\riem'$ be as in \emph{Corollary \ref{co:density_general_isotopy}}. 
 Assume that $\riem'$ (and hence $\riem$) is a bordered Riemann surface of genus $g$ and $n$ borders with $n \geq 1$.   
 Then the set of restrictions of elements of $A(\riem')$ to $\riem$ is dense in $A(\riem)$.  
\end{theorem}
\begin{proof}
 Let $R'$ be the double of $\riem'$.  It is a surface of genus $2g+n-1$, so the dimension of $A(R')$ is $2g+n-1$.  Let $a_1,\ldots,a_{2g + n-1}$ denote a set of generators for the fundamental group of $\riem$.   Given any $\alpha \in A(\riem)$, there is a $\beta \in A(R')$ with the same periods.  Thus $\alpha - \beta$ is exact on $\riem$, with primitive $H$ say.    
 
 Thus by Corollary \ref{co:density_general_isotopy}, for any $\epsilon >0$ there is an
 $h \in \mathcal{D}(\riem')$ such that 
 \[   \left\| H- \left. h \right|_{\riem} \right\|_{\mathcal{D}(\riem)}< \epsilon.  \] 
 Setting $\delta = \partial h + \left.\beta \right|_{\riem'} \in A(\riem')$
 we have that 
 \[  \left\| \alpha - \left. \delta\right|_{\riem} \right\|_{A(\riem)} = \left\| H - \left. h\right|_{\riem} \right\|_{\mathcal{D}(\riem)} < \epsilon.  \]
 This completes the proof.
\end{proof}

The following example shows that the truth of Corollary \ref{co:density_general_isotopy} depends on the fact that every component of $R \backslash \riem$ contains a component of $R \backslash \riem'$.   
Let $R = \sphere$.  Fix $r \in (0,1)$ and let $\riem = \{  z : r < |z| < 1 \}$.  For $0<r'<r$ and $s'>1$, Theorem \ref{th:central_density_theorem} says that for $\riem' = \{ z: r'< |z| <s' \}$, 
$\mathcal{D}(\riem')$ is dense in $\mathcal{D}(\riem)$.  However, setting instead $\riem' = \disk$, it is not true that $\mathcal{D}(\disk)$ is dense in $\mathcal{D}(\riem)$.  To see this, fix $z \in \mathbb{C}$ such that $|z| >1$
and observe that the functional on $\mathcal{D}(\riem)$ given by 
\[     \Lambda(h) = \left. J_q(\Gamma)h  \right|_z       \]
is bounded, since point evaluation is bounded on the Dirichlet space of $\{ z : |z| >1 \} \cup \{ \infty \}$.  This functional vanishes on $\mathcal{D}(\disk)$ but not on the entire space $\mathcal{D}(\riem)$.  Thus $\mathcal{D}(\disk)$ is not dense in $\mathcal{D}(\riem)$.   

Also, even removing a point from a component is not enough.    Let $\riem' = \mathbb{D} \backslash \{0 \}$ and let $\riem$ be as above. Observe that any element of $\mathcal{D}(\disk \backslash \{0\})$ extends to an element of $\mathcal{D}(\disk)$.  Since $\mathcal{D}(\disk)$ is not dense in $\mathcal{D}(\riem)$ by the previous paragraph,
the Theorem \ref{th:central_density_theorem} does not extend to this case.  On the other hand, the set of restrictions of holomorphic functions on $\disk \backslash \{0 \}$ to $\riem$ is dense in $\mathcal{D}(\riem)$.   

We now prove Theorem \ref{th:dirichlet_density_squeeze} and Corollary \ref{co:embedding_in_double}.  

\begin{proof}(of Theorem \ref{th:dirichlet_density_squeeze})
 The set of restrictions of elements of $\mathcal{D}(\riem'')$ to $\riem$ contains the set of restrictions of elements of $\mathcal{D}(\riem')$ to $\riem$.  Thus by Corollary \ref{co:density_general_isotopy}, the set of restrictions of elements of $\mathcal{D}(\riem'')$ to $\riem$ are dense.  
 
 Similarly, by Theorem \ref{th:form_general_isotopy_density}, the set of restrictions of elements of $A(\riem'')$ to $\riem$ is dense in $A(\riem)$. 
\end{proof}
\begin{proof}(of Corollary \ref{co:embedding_in_double}).
 Observe that $\riem'$ can be viewed as a subset of its double $\riem^D$, and its boundary can be identified with $n$ analytic curves in the double.  Thus the claim follows from Theorem \ref{th:dirichlet_density_squeeze} applied with $R = \riem^D$. 
\end{proof}

 We indicate another possible approach to proving Theorem \ref{th:form_general_isotopy_density} (and therefore Theorem \ref{th:dirichlet_density_squeeze}), using the result of Askaripour and Barron \cite{AskBar}.  We assume that $\Gamma_k$ and $\Gamma_k'$ are analytic curves for $k=1,\ldots,n$.   Assume also that the universal cover of $R$ is the disk.  Let $\pi:\disk \rightarrow R$ be the covering map.  Choose a collection of curves $\gamma_j$, $j=1,\ldots, g$ dissecting the compact surface $R$, where $g$ is the genus, to obtain a fundamental polygon $F$ in the disk $\disk$.  Choose the dissection such that every curve $\Gamma_k$ and $\Gamma_k'$ is crossed by at least one of the dissecting curves.  In that case, the sets $\pi^{-1}(\riem) \cap F$ and $\pi^{-1}(\riem') \cap F$ will be Carath\'eodory sets in the plane, and one can apply \cite[Proposition 2.1]{AskBar} to obtain the result in the case of analytic curves.
 
 One would need to show that such a dissection exists in general, which should not pose much difficulty.  However, if one attempts this argument in the case of  quasicircles, then establishing that the dissecting curves can be made to have the intersection property might be a delicate problem.  On the other hand, if the dissecting curves are chosen not to intersect $\Gamma_k$ and $\Gamma_k'$, the lifted sets $\pi^{-1}(\riem) \cap F$ and $\pi^{-1}(\riem') \cap F$ would not be Carath\'eodory sets, and one could not apply their result directly.  
 
 It should be noted that \cite[Proposition 2.1]{AskBar} does not require analytic conditions on the boundary of $\riem$ and $\riem'$, as we do in Theorem \ref{th:form_general_isotopy_density}.  Although we were able to remove the restrictions on the outer domain $\riem''$ to some extent in Theorem \ref{th:dirichlet_density_squeeze}, we did not do so for $\riem$ itself.  Thus their result suggests that the analytic conditions of Theorem \ref{th:dirichlet_density_squeeze} can be weakened.
\end{subsection}	
\begin{subsection}{The Schiffer comparison operator for open surfaces}
 Next we define a certain comparison operator, which generalizes an operator considered by Schiffer \cite{Courant_Schiffer}.  
 
 Let $R$ be a compact Riemann surface.  Let $\riem$ and $\riem'$ be Riemann surfaces such that $\riem \subseteq \riem' \subseteq R$, and such that the inclusion maps from $\riem$ into $\riem'$ and $\riem'$ into $R$ are holomorphic.  
 Define the restriction operator 
 \begin{align*}
  R(\riem',\riem): A(\riem') & \rightarrow A(\riem) \\
   \alpha & \mapsto \left. \alpha \right|_{\riem}.  
 \end{align*}
 
 Assume that $\riem'$ has a Green's function $g_{\riem'}$.  We then define the Bergman kernel of $\riem'$ to be
 \[  K_{\riem'} = - \frac{1}{\pi i} \partial_z \overline{\partial}_{w} g_{\riem'}(w,z).  \]
 The Schiffer comparison operator is then defined to be 
 
 \begin{align}
  S(\riem,\riem') : A(\riem) & \rightarrow A(\riem') \nonumber \\
  \alpha & \mapsto \iint_{\riem} K_{\riem'}(z,w) \wedge_w \alpha(w).   
 \end{align}
 
 We then have the following result, which strangely seems to have been missed by Schiffer, even in the planar case.  By a hyperbolic metric, we mean a complete, constant negative curvature metric.   
 \begin{theorem}  \label{th:restriction_adjoint}
  Let $R$ be a compact Riemann surface, $\riem$ and $\riem'$ be Riemann surfaces such that $\riem \subset \riem' \subset R$, {{$\mathrm{cl}\riem\subset \riem'$ and}} the inclusion maps from $\riem$ to $\riem'$ and $\riem'$ to $R$ are holomorphic.  Assume that $\riem'$ has a Green's function $g_{\riem'}(w,z)$, and that $\riem'$ possesses a hyperbolic metric.    Denoting the adjoint of $R(\riem',\riem)$ by $R(\riem',\riem)^*$ we have
  \[ S(\riem,\riem') = R(\riem',\riem)^*. \] 
 \end{theorem}
\begin{proof}   Let $\alpha \in A(\riem)$ and $\beta \in A(\riem')$.  Then using 
	the reproducing property of the Bergman kernel $K_{\riem'}$ and assuming that we are allowed to interchange the order of integration we have
 \begin{align*}
  \left( S(\riem,\riem') \alpha, \beta \right)_{A(\riem')} & = \frac{i}{2} \iint_{\riem',z} \iint_{\riem,w} 
   K_{\riem'} (z,w) \wedge_w \alpha(w) \wedge_z \overline{\beta(z)} \\
   & = - \frac{i}{2} \iint_{\riem,w} \iint_{\riem',z} \overline{K_{\riem'} (w,z)} \wedge_z \overline{\beta(z)} \wedge_w \alpha(w) \\
   & =- \frac{i}{2} \iint_{\riem} \overline{\beta(w)}\wedge_w  \alpha(w) \\
   & = \left( \alpha, R(\riem',\riem) \beta  \right)_{A(\riem)}.  
 \end{align*}

To justify the change of the order of integration, let  $\riem$ be compactly included in $\riem'$ and $K_n$ be a sequence of compact subsets of $\riem$ that exhaust it (i.e. $K_n\to {\riem}$). { Denote the $L^p$ norm over a set $U$ with respect to the hyperbolic metric on $\riem'$ by $\| \cdot \|_{p,U}$ (see \cite{RSS_L2Beltrami}).  Note that the $L^2$ norm of a one-form (a one-differential in the terminology of \cite{RSS_L2Beltrami}) with respect to the hyperbolic metric agrees with the $L^2$ norm used in this paper.  Now for fixed $w$ set 
\[  c_n(w) = \| K_{\riem'}(\cdot,w) \|_{1,K_n}.   \]
Note that $c_n(w)$ is a one-form on $\riem$ for every $n$ (of the form $a(w) d\bar{w}$ in local coordinates).  
Then 
\begin{equation*}
 \iint_{\riem,w} \iint_{K_n,z} |K_{\riem'}(z,w) \wedge_w \alpha(w) \wedge_z \overline{\beta(z)} | 
  \leq \| \beta \|_{\infty,K_n} \iint_{\riem,w} | c_n(w) \wedge_w \alpha(w)|
\end{equation*}
so setting $C_n = \| c_n(\cdot) \|_{\infty,\riem}$ we have 
\begin{equation*}
    \iint_{\riem,w} \iint_{K_n,z} |K_{\riem'}(z,w) \wedge_w \alpha(w) \wedge_z \overline{\beta(z)} |  
    \leq C_n \| \alpha \|_{1,\riem} \| \beta \|_{\infty,K_n}.
\end{equation*}  
Now for the compact set $K_n$, a standard argument shows that there are constants $D_n$ such that  
$\| \beta \|_{\infty,K_n} \leq D_n \| \beta \|_{2,\riem}$ 
(see e.g. \cite[Lemma 2.1]{Schiffer_comparison}, which can be made global as in \cite[Lemma 4.7]{RSS_L2Beltrami}).  Now the norm of the characteristic function on $\riem'$ is finite so there is a constant $E$ such that 
$\| \alpha \|_{1,\riem} \leq E \| \alpha \|_{2,\riem} \leq E \| \alpha \|_{2,\riem'}$.   Thus
\begin{equation*}
 \iint_{\riem,w} \iint_{K_n,z} |K_{\riem'}(z,w) \wedge_w \alpha(w) \wedge_z \overline{\beta(z)} |  \leq C_n D_n E \| \alpha \|_{2,\riem'} \| \beta \|_{2,\riem}.
\end{equation*}

}

So Fubini-Tonelli's theorem applies and we have
 \begin{equation}\label{fubini}
 \iint_{K_n,z} \iint_{\riem,w} 
   K_{\riem'} (z,w) \wedge_w \alpha(w) \wedge_z \overline{\beta(z)} \\
    = -\iint_{\riem,w} \iint_{K_n,z} \overline{K_{\riem'} (w,z)} \wedge_z \overline{\beta(z)} \wedge \alpha(w)
    \end{equation}

Now since $\beta \chi_{K_n} \to \beta \chi_{\riem}$ in $L^2$ norm and the operator with kernel $K_\riem$ (or $\overline{K_{\riem}}$) was bounded on $L^2(\riem)$ then 
$$\iint_{K_n,z}
   \overline{K_{\riem'} (w,z)}  \wedge_z \overline{\beta(z)} \to \iint_{\Sigma} 
   \overline{K_{\riem'} (w,z)}\wedge_z \overline{\beta(z)},$$
in the $L^2(\riem)-$norm. This, and \eqref{fubini} together with the fact that $\alpha \in L^2(\riem)$ (the dual of $L^2(\riem)$) yields that 

$$\iint_{\riem',z} \iint_{\riem,w} K_{\riem'} (z,w) \wedge_w \alpha(w) \wedge_z \overline{\beta(z)}
    = - \iint_{\riem,w} \iint_{\riem',z} \overline{K_{\riem'} (w,z)} \wedge_z \overline{\beta(z)} \wedge_w \alpha(w).$$ 
\end{proof}

We now can prove the final result.
\begin{proof}(of Theorem \ref{th:Bergman_comparison_dense}).  
 {Under the hypotheses, $\riem'$ possesses a hyperbolic metric.}
 It is obvious that $\text{Ker}(R(\riem',\riem))$ is zero.  Thus, 
 \[ \text{cl} \, \text{Im} (S(\riem,\riem')) = \text{Ker} (R(\riem',\riem) )^\perp = A(\riem').  \]
 
 For the kernel, observe that
 by Corollary \ref{co:density_general_isotopy}, the image of $R(\riem',\riem)$ is dense in $A(\riem)$.  Thus 
 \[  \text{Ker} (S(\riem,\riem') )= (\text{cl} \, \text{Im} (R(\riem',\riem)))^\perp  
  = \{ 0 \}.    \]
  This completes the proof.
\end{proof}
\begin{remark}
Note that operator $S(\Sigma, \Sigma')$ can not have closed range, because that would imply that $R(\Sigma', \Sigma)$ has closed range.  Since $R(\Sigma', \Sigma)$ is also injective this would imply that it is surjective, which is clearly not the case. 
\end{remark}
\end{subsection}
\end{section}


\begin{thebibliography}{99}
 \bibitem{Ahlfors_Sario} Ahlfors, L. V.; Sario, L.
 { Riemann surfaces}.
   {\it Princeton Mathematical Series}, No. 26 Princeton University Press, Princeton, N.J. 1960.
 \bibitem{AskBar} Askaripour, N.; Barron, T. On extension of holomorphic k-differentials on open Riemann surfaces.   {\it  Houston J. Math.} {\bf 40} (2014), no. 4, 117--1126.
 \bibitem{BergmanSchiffer} Bergman, S.; Schiffer, M. {Kernel functions and conformal mapping.}    {\it Compositio Math.} {\bf 8} (1951), 205--249.
\bibitem{Courant_Schiffer}  Courant, R.  {Dirichlet's principle, conformal mapping, and minimal surfaces.} With an appendix by M. Schiffer. Reprint of the 1950 original. Springer-Verlag, New York-Heidelberg, 1977.
  

\bibitem{Gakhov_book} Gakhov, F. D. Boundary value problems. Translated from the Russian. Reprint of the 1966 translation. Dover Publications, Inc., New York, 1990.
\bibitem{GauthierSharifi_Luzin} Gauthier, P. M.; Sharifi, F.  Luzin-type harmonic approximation on subsets of non-compact Riemannian manifolds. {\it J. Math. Anal. Appl.} {\bf 474} (2019), no. 2, 1132--1152.
\bibitem{Nap_Yulm} Napalkov, V. V., Jr.; Yulmukhametov, R. S. On the Hilbert transform in the Bergman space. (Russian) Mat. Zametki {\bf 70} (2001), no. 1, 68--78; {\it Translation in Math. Notes} {\bf 70} (2001), no. 1-2, 61--70.
  Mathematica/Mathematische Lehrb\"ucher, Band XXV. Vandenhoeck \& Ruprecht,
  G\"ottingen  1975.
\bibitem{RS_monster} Radnell, D. and Schippers, E. Quasisymmetric sewing in rigged Teichm\"uller space. {\it Commun. Contemp. Math.} {\bf 8} (2006), no. 4, 481--534.
\bibitem{RSS_L2Beltrami}  Radnell, D; Schippers, E; and Staubach, W. Quasiconformal maps of bordered Riemann surfaces with $L^2$ Beltrami differentials. {\it J. Analyse Math.} {\bf 132} (2017), 229--245. 
 \bibitem{Royden}  Royden, H. L. {Function theory on compact Riemann surfaces}. {\it J. Analyse Math.} {\bf 18}, (1967), 295--327. 
 \bibitem{Rodin_book}  Rodin, Yu. L. The Riemann boundary problem on Riemann surfaces. Mathematics and its Applications (Soviet Series), {\bf 16}. D. Reidel Publishing Co., Dordrecht, 1988.
\bibitem{Schiffer_first} Schiffer, M. {The kernel function of an orthonormal system}. {\it Duke Math. J.} {\bf 13}, (1946). 529 -- 540.
 \bibitem{Schiffer_Spencer} Schiffer, M. and Spencer, D.  {Functionals on finite Riemann surfaces}. Princeton University Press, Princeton, N. J., 1954. 
 \bibitem{Schippers_Staubach_scattering}   Schippers, E.; and Staubach, W. { Riemann boundary value problem on quasidisks, Faber isomorphism and Grunsky operator}. {\it Complex Anal. Oper. Theory} {\bf 12} (2018) no. 2, 325--354.
 \bibitem{Schippers_Staubach_general_transmission} Schippers, E. and Staubach, W.  {Transmission of harmonic functions through quasicircles on compact Riemann surfaces}.  
  arXiv:1810.02147v2
 \bibitem{Schiffer_comparison}  Schippers, E. and Staubach, W.   Plemelj-Sokhotski isomorphism for quasicircles in Riemann surfaces and the Schiffer operators (2018). To appear in {\it Mathematische Annalen.}
 \bibitem{Sharifi_thesis} Sharifi, F.  Uniform Approximation on Riemann Surfaces  (2016).  Electronic Thesis and Dissertation Repository. 3954.
https://ir.lib.uwo.ca/etd/3954
 \bibitem{ShenFaber} Shen, Y.  {Faber polynomials with applications
to univalent functions with quasiconformal extensions}.  {\it Sci. China Ser}. A {\bf 52} (2009), no. 10, 2121–-2131.
\end{thebibliography}
\end{document}